\documentclass{article}
\usepackage{amsmath, amsthm, latexsym, amssymb, graphicx,color,cite,setspace}
\theoremstyle{theorem}
\newtheorem{theorem}{Theorem}[section]
\newtheorem{lemma}[theorem]{Lemma}
\newtheorem{definition}[theorem]{Definition}

\newtheorem{proposition}[theorem]{Proposition}
\newtheorem{convention}[theorem]{Convention}
\theoremstyle{remark}
\newtheorem{remark}{Remark}

\renewcommand\Re{\operatorname{Re}}
\renewcommand\Im{\operatorname{Im}}
\author{Evan Randles\thanks{This material is based upon work supported by the National Science Foundation Graduate Research Fellowship under Grant No. DGE-0707428}\ \thanks{Corresponding author: edr62@cornell.edu}\\
\normalsize  Center for Applied Mathematics\\[-0.8ex]
\normalsize Cornell University\\
\and
Laurent Saloff-Coste\thanks{This material is based upon work supported by the National Science Foundation under Grant No. DMS-1004771}\\
\normalsize Department of Mathematics\\[-0.8ex]
\normalsize Cornell University\\
}
\title{On the convolution powers of complex functions on $\mathbb{Z}$}
\date{}
\begin{document}
\maketitle

\begin{abstract} 
The local limit theorem describes the behavior of the convolution powers of a probability distribution supported on $\mathbb{Z}$. In this work, we explore the role played by positivity in this classical result and study the convolution powers of the general class of complex valued functions finitely supported on $\mathbb{Z}$. This is discussed as de Forest's problem in the literature and was studied by Schoenberg and Greville. Extending these earlier works and using techniques of Fourier analysis, we establish asymptotic bounds for the sup-norm of the convolution powers and prove extended local limit theorems pertaining to this entire class. As the heat kernel is the attractor of probability distributions on $\mathbb{Z}$, we show that the convolution powers of the general class are similarly attracted to a certain class of analytic functions which includes the Airy function and the heat kernel evaluated at purely imaginary time.
\end{abstract}

\noindent{\small\bf Keywords:} Convolution powers, Local limit theorems.\\

\noindent{\small\bf Mathematics Subject Classification:} Primary 42A85; Secondary  60F99.

\section{Introduction}\label{introsec}

\noindent Let $\phi:\mathbb{Z}\rightarrow \mathbb{C}$ be a finitely supported
function. We wish to study the convolution powers of $\phi$, that is, the functions $\phi^{(n)}:\mathbb{Z}\rightarrow\mathbb{C}$
defined iteratively by 
\begin{equation*}
 \phi^{(n)}(x)=\sum_{y\in \mathbb{Z}}\phi^{(n-1)}(x-y)\phi(y),
\end{equation*}
where $\phi^{(1)}=\phi$. This study has been previously motivated by problems in data smoothing and numerical difference schemes for partial differential equations \cite{TNEG,IJS,VT1,VT2}. We encourage the reader to see the recent article \cite{DSC1} for background discussion and pointers to the literature.\\

\noindent In the case that the support of $\phi$ is empty or
contains a single point, the convolution powers of $\phi$ are rather easy to
describe. The present article focuses on functions $\phi$ with finite
support consisting of more than one point; in this case we say
that the support of $\phi$ is \textit{admissible}. When the function $\phi$ is a probability distribution, i.e., it is
non-negative and satisfies
\begin{equation*}
\sum_{x\in\mathbb{Z}}\phi(x)=1,
\end{equation*}
the behavior of $\phi^{(n)}$ for large values of $n$ is well-known and is the subject of the local limit theorem. A modern treatment of this classical result can be found in Chapter $2$ of \cite{LL} (see also Chapter $2$ of \cite{FS}).  Our aim is to extend the results of \cite{DSC1} and describe the limiting behavior for the
general class of complex valued functions on $\mathbb{Z}$ with admissible
support. In particular, we give bounds on the supremum of $|\phi^{(n)}|$
and prove ``generalized'' local limit theorems.\\
 
\noindent As an example, we consider the function $\phi:\mathbb{Z}\rightarrow \mathbb{C}$ defined by
\begin{equation}\label{ex1def}
 \phi(0)=\frac{1}{8}(5-2i)\hspace{1cm}\phi(\pm 1)=\frac{1}{8}(2+i)\hspace{1cm}\phi(\pm 2)=-\frac{1}{16}\hspace{1cm}
\end{equation}
and $\phi=0$ otherwise. The convolution powers $\phi^{(n)}$ for $n=100,1000,10000$ are illustrated in Figures \ref{fig:ex1abs} and \ref{fig:ex1real}. We make two crucial observations about these graphs: First, it appears that the supremum $\|\phi^{(n)}\|_{\infty}$ is decaying on the order of $n^{-1/2}$; this is consistent with the classical theory for probability distributions. Second, as $n$ increases, $|\phi^{(n)}(x)|$ appears to be constant on increasingly large intervals centered at $0$. This is in stark contrast to the behavior described by the classical local limit theorem for probability distributions.  In the present article, we prove that there are constants $C,C'>0$ for which
\begin{equation*}
Cn^{-1/2}\leq \|\phi^{(n)}\|_{\infty}\leq C'n^{-1/2}.
\end{equation*}
We also show that
\begin{equation*}
\phi^{(n)}(\lfloor xn^{1/2}\rfloor)=\frac{n^{-1/2}}{\sqrt{4\pi
i/8}}e^{-8|x|^2/4i}+o(n^{-1/2})
\end{equation*}
for $x$ in any compact subset of $\mathbb{R}$. Here, $\lfloor \cdot\rfloor$ denotes the greatest integer function. We note that this approximation cannot hold uniformly for all $x\in\mathbb{R}$ because the modulus of $(4\pi i/8)^{-1/2}\exp(-8|x|^2/4i)$  is a non-zero constant whereas $\phi^{(n)}$ has finite support for each $n$. For comparison with Figure
\ref{fig:ex1real}, Figure \ref{fig:ex1realtheory}
shows the graph of $\Re((4\pi ni/8)^{-1/2}\exp(-8|x|^2/4ni))$ for
$n=100,1000,10000$. We will return to this example in Subsection \ref{ex1} and justify the claims made above.\\

\vspace{1in}

\begin{figure}[h!]
\centering\includegraphics[width=5in]{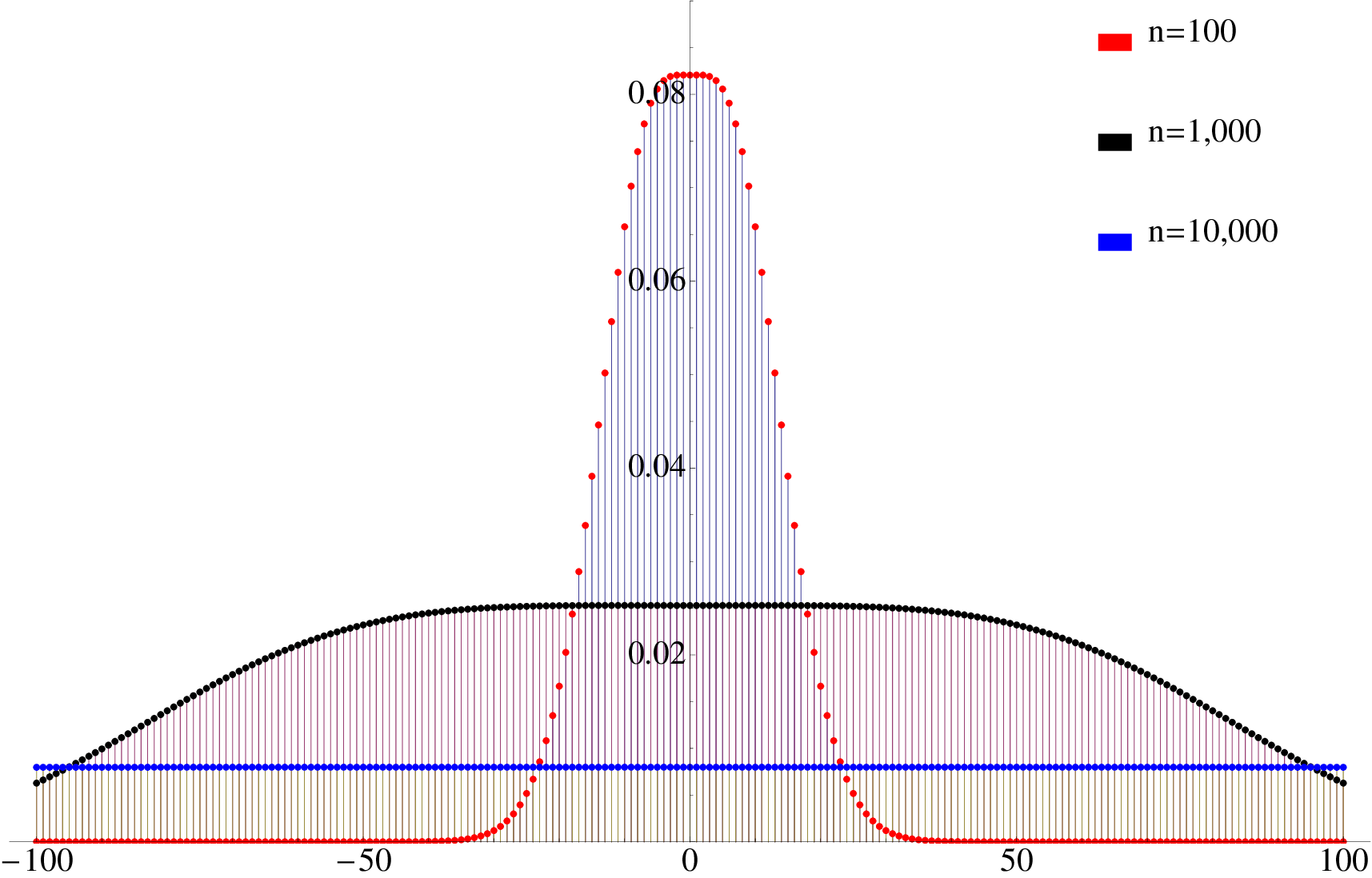}
\caption{$|\phi^{(n)}|$ for $n=100,1000,10000$}
\label{fig:ex1abs}
\end{figure}

\begin{figure}[h!]
\centering\includegraphics[width=5in]{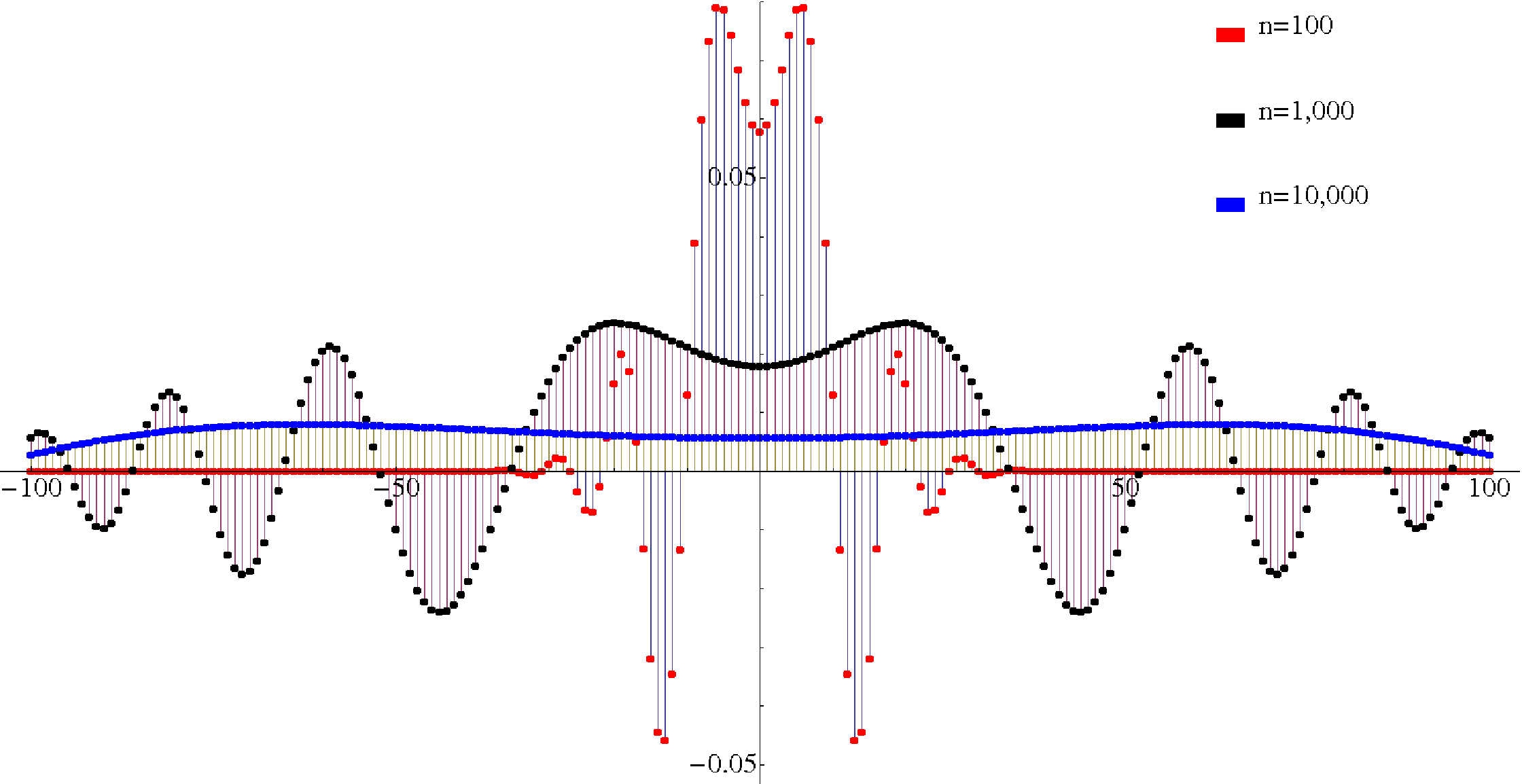}
 \caption{$\Re(\phi^{(n)})$ for $n=100, 1000, 10000$}
\label{fig:ex1real}
\end{figure}

\begin{figure}[h!]
\centering\includegraphics[width=5in]{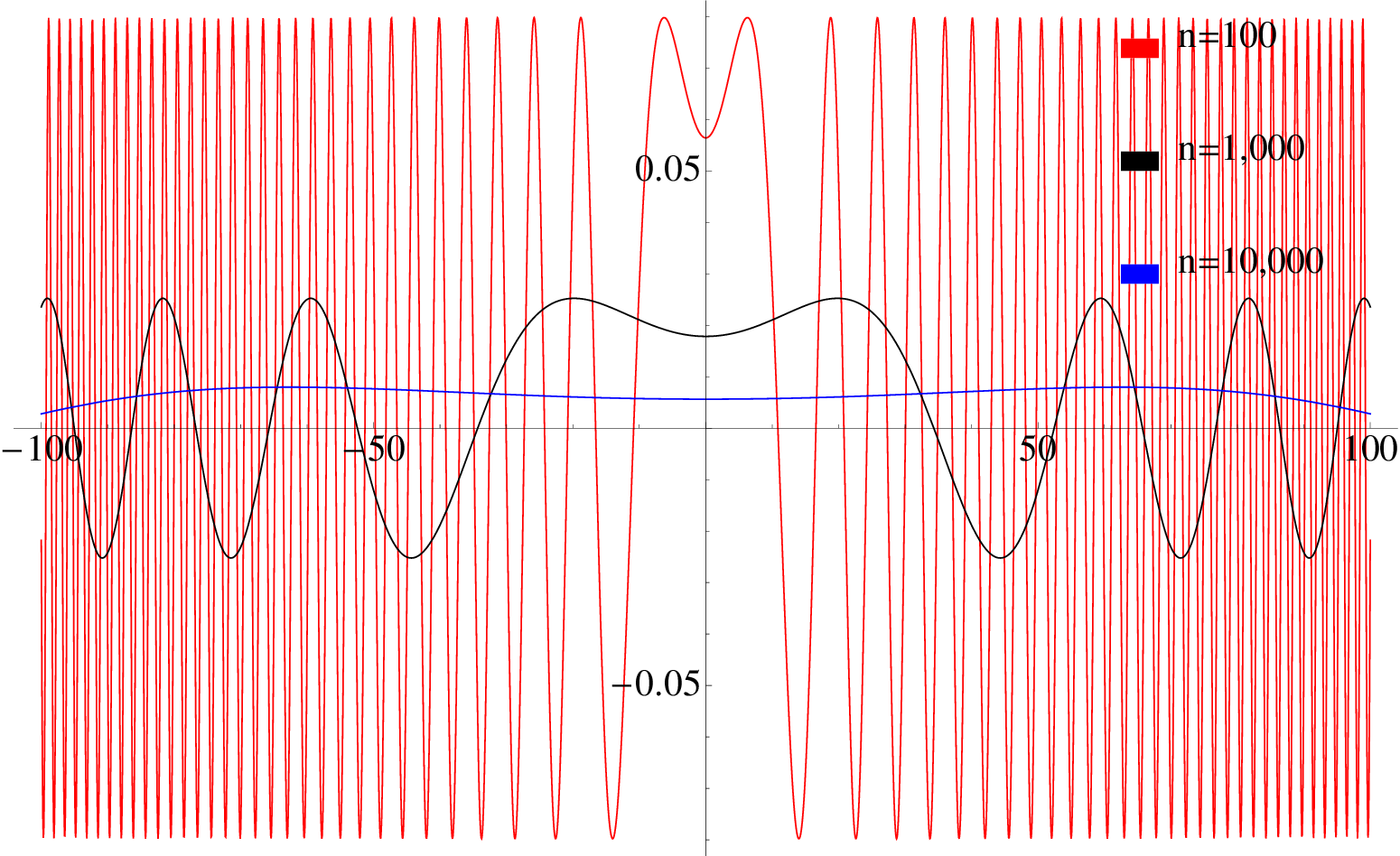}
 \caption{$\Re((4\pi in/8)^{-1/2}\exp(-8|x|^2/4ni))$ for $n=100, 1000,
10000$}
\label{fig:ex1realtheory}
\end{figure}

\vspace{1in}

\noindent The Fourier transform is central to the arguments made in this article. We recall its definition: For $\phi:\mathbb{Z}\rightarrow
\mathbb{C}$, finitely supported, the Fourier transform of $\phi$ is the function $\hat{\phi}:\mathbb{R}\rightarrow\mathbb{C}$ defined by
\begin{equation}\label{fouriertransform}
\hat{\phi}(\xi)=\sum_{x\in\mathbb{Z}}\phi(x)e^{ix\xi}
\end{equation}
for $\xi\in\mathbb{R}$.\\

\noindent Our first main result is illustrated in the following theorem.
\begin{theorem}\label{mainbound}
Let $\phi:\mathbb{Z}\rightarrow \mathbb{C}$ have admissible support and let 
$A=\sup_\xi|\hat{\phi}(\xi)|.$ Then there is a natural number $m\geq 2$, and positive constants $C$ and $C'$ such that
\begin{equation}\label{mainboundsupeq}
Cn^{-1/m}\leq A^{-n}\|\phi^{(n)}\|_{\infty}\leq C'n^{-1/m}
\end{equation}
for all natural numbers $n$.
\end{theorem}

\begin{remark}
The natural number $m\geq 2$ appearing in Theorem \ref{mainbound} is consistent with those appearing Theorems in \ref{mainstrong} and \ref{mainweak}; upon dividing $\phi$ by $A$, it is defined by \eqref{defofmphi}.
\end{remark}

\noindent In the classical local limit theorem, the convolution powers of a
probability distribution are approximated by the heat kernel, an analytic
function. In the present setting, the convolution
powers $\phi^{(n)}$ are analogously approximated by certain analytic
functions. We now define these so-called \emph{attractors}: Let $m\geq 2$ be a natural
number and $\beta$ be a non-zero complex number for which $\Re(\beta)\geq 0$. We define $H_m^{\beta}:\mathbb{R}\rightarrow\mathbb{C}$ by
\begin{equation}\label{Hdef}
H_m^{\beta}(x)=\frac{1}{2\pi}\int_{\mathbb{R}}e^{-i xu}e^{-\beta u^m}du
\end{equation}
provided the integral converges as an improper Riemann integral. If additionally, for each $\epsilon>0$ there exists $M_\epsilon>0$ such that
\begin{equation*}
\left|H_m^{\beta}(x)-\frac{1}{2\pi}\int_M^Me^{-ixu}e^{-\beta u^m}\,du\right|<\epsilon
\end{equation*}
for all $M\geq M_\epsilon$ and $x\in S\subseteq\mathbb{R}$, we say that the integral defining $H_m^{\beta}$ converges uniformly in $x$ on $S$. When $\Re(\beta)>0$  and $m$ is an even natural number, it is easy to see that 
\begin{equation*}
 |e^{-i xu}e^{-\beta u^m}|= e^{-\Re(\beta)u^m}\in L^1(\mathbb{R})
\end{equation*}
whence the defining integral converges uniformly in $x$ on $\mathbb{R}$. In this case, $H_m^{\beta}$ is equivalently defined by its inverse Fourier transform, $e^{-\beta u^m}$. In the case that $\Re(\beta)=0$, it is not immediately clear for which values of $m$ or in what sense the integral in \eqref{Hdef} will converge. It will be shown that when $m>2$, the integral converges uniformly in $x$ on $\mathbb{R}$ and, when $m=2$, it converges uniformly in $x$ on any compact set. This is the subject of Proposition \ref{Hconverge}. The proposition extends the results of \cite{TNEG} in which only odd values of $m$ (for $\Re(\beta)=0$) were considered. \\

\noindent In the case that $m\geq 2$ is even and $\Re(\beta)\geq 0$,  $H_m^{\beta}$ is 
the integral kernel of the holomorphic semigroup $T_\beta= e^{-\beta (\Delta)^{m/2}}$ generated by the non-negative self-adjoint operator $(\Delta)^{m/2}$ on $L^2(\mathbb{R})$; here, $\Delta$ is the unique self-adjoint extension of $-(d/dx)^2$ originally defined on smooth compactly supported functions on $\mathbb{R}$. In the specific case that $m=2$,
\begin{equation}\label{heatker}
 H_2^{\beta}(x)=\frac{1}{\sqrt{4\pi \beta}}e^{-\frac{|x|^2}{4\beta}}
\end{equation}
is the heat kernel evaluated at complex time $\beta$. There is an extensive theory concerning these semigroups and generalizations thereof for $\Re(\beta)>0$. In the context of $\mathbb{R}^d$,  we refer the reader to the articles \cite{Davies1995,BarbatisDavies1995} which consider general self-adjoint operators with measurable coefficients, called superelliptic operators, each comparable to $(\Delta)^{m/2}$ for some even $m\geq 2$. In the context of Lie groups, such generalizations are treated by \cite{Robinsonbook,Robinson1991, Dungey2002}. An integral piece of this theory concerns off-diagonal estimates for these kernels. In our setting, this is the estimate \begin{equation}\label{Hassymptoticnicecase}
|H_m^{\beta}(x)|\leq C\exp(-B|x|^{\frac{m}{m-1}})
\end{equation}
for all $x\in\mathbb{R}$, where $C,B>0$. Given \eqref{Hdef}, a complex change of variables via contour integration followed by a minimization argument easily yields the estimate \eqref{Hassymptoticnicecase} (see Proposition 5.3 of \cite{Robinsonbook}).\\

\noindent Viewing things from a slightly different perspective, when $m\geq 2$ is even and $\Re(\beta)>0$, the function $Z:(0,\infty)\times\mathbb{R}\rightarrow\mathbb{C}$, defined by
\begin{equation*}
Z(t,x)=H_m^{t\beta}(x),
\end{equation*}
is a fundamental solution to the constant-coefficient parabolic equation   
\begin{equation*}
\frac{\partial}{\partial t}+i^m\beta\frac{\partial^m}{\partial x^m}=0.
\end{equation*}
The treatise \cite{Eidelmanbook1969} by S. D. Eidelman gives an extensive treatment of such ``higher order'' parabolic equations with variable coefficients on $\mathbb{R}^d$. For second order parabolic systems ($m=2$), A. Friedman's classic text \cite{Friedmanbook1964} is an excellent reference. 

\begin{remark}
In the case that $\Re(\beta)>0$ and $m$ is even, the
function $H_m^{\beta}$ and the function $H_{m,b}$ used in Theorem 2.3 of
\cite{DSC1} and defined by its Fourier transform,
$\hat{H}_{m,b}(\xi)=e^{-(1+ib)\xi^m},$ are connected via the relation
\begin{equation*}
H_{m,\frac{\Im(\beta)}{\Re(\beta)}}\left(\frac{x}{(\Re(\beta))^{1/m}}
\right)=(\Re(\beta))^{1/m}H_m^{\beta}(x)\\
\end{equation*}
which follows from the change of variables $u\mapsto (\Re(\beta))^{1/m}u$.
\end{remark}

\noindent In the case that $m\geq 2$ is even and $\beta>0$, the functions $H_m^{\beta}$ are real valued and when $m>2$ they take on both positive and negative values. As the classical Wiener measure is defined by the transition kernel $H_2^{1}$, V. Krylov \cite{Krylov1960} and later K. Hochberg \cite{Hochberg1978} considered finitely additive signed measures on path space defined by $H_m^{1}$ for $m\in\{4,6,8,\dots\}$. Recently, D. Levin and T. Lyons \cite{LevinLyons2009} used rough path theory to study these measures. Both Krylov and Hochberg associated something like a process to such finitely additive measures, called signed Wiener measures in \cite{Hochberg1978}, to mimic the way that Brownian motion is associated to Wiener measure. This theory has been pursued recently by a number of authors \cite{HochbergOrsingher1996, Lachal2007, Lachal2012, Nishioka1996, Sato2002}, and such ``processes'' are now called pseudo-processes; the pseudo-process corresponding to $H_4^1$ is called the biharmonic pseudo-
process. We do not pursue signed Wiener measures or pseudo-processes here.\\

\begin{figure}[h!]
\centering\includegraphics[width=5in]{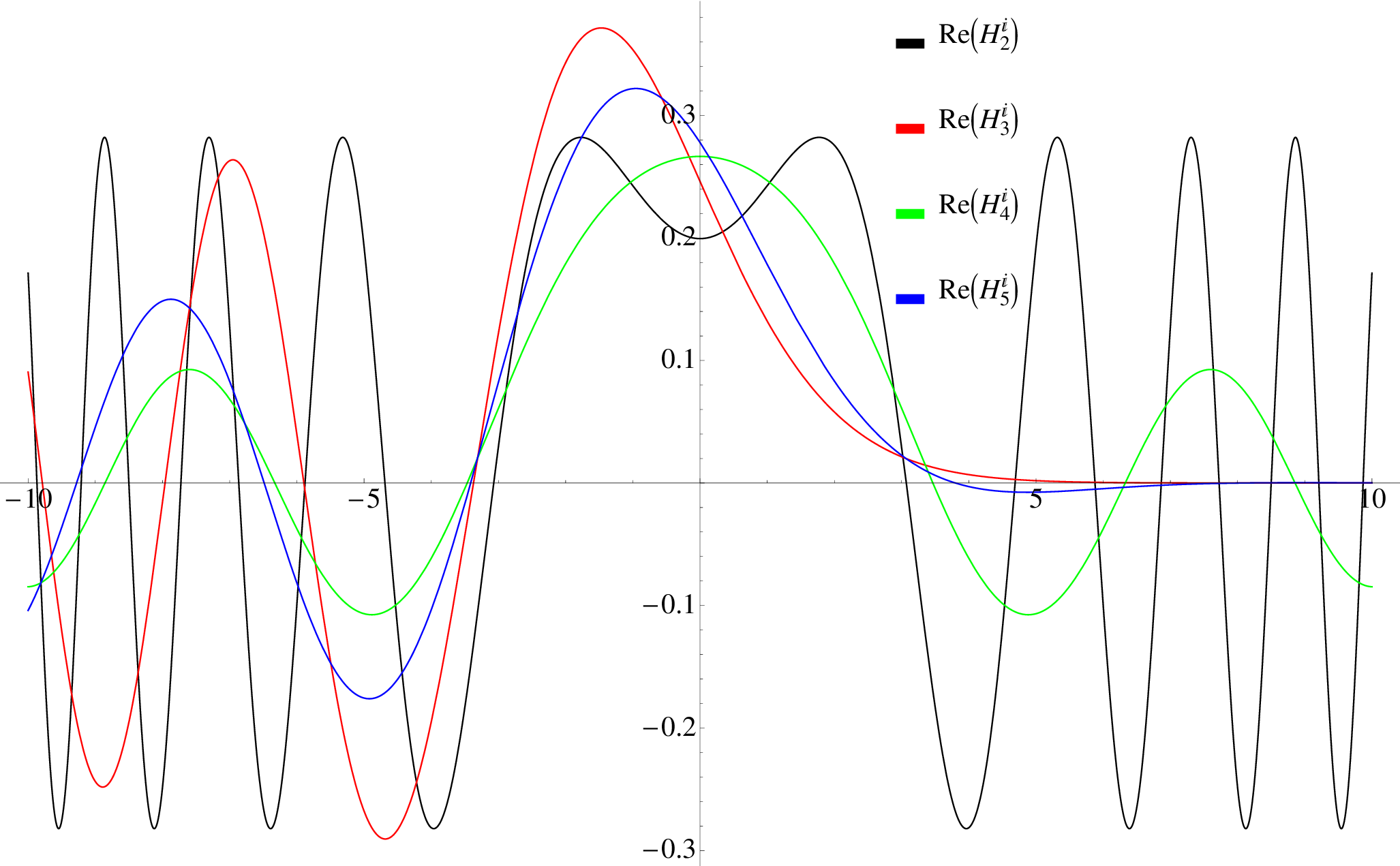}
\caption{$\Re(H_m^i(x))$ for $m=2,3,4,5$}
\label{fig:Hgraph}
\end{figure}

\noindent When $\beta$ is purely imaginary and $m\geq 2$, the situation is very different from those described above. The graphs of $\Re(H_m^i(x))$ for $m=2,3,4,5$ are illustrated in Figure \ref{fig:Hgraph}. When $\beta=i/m$, $H_m^{\beta}=H_m^{i/m}$ satisfies the ordinary differential equation
\begin{equation*}
\frac{d^{m-1}y}{dx^{m-1}}+(-i)^{m-1}xy=0,
\end{equation*}
c.f., Remark $3$ of \cite{RH}. When $m=3$, this is Airy's equation and $H_3^{i/3}(x)$ is the famous Airy function, $\text{Ai}(x)$. The study of the functions $H_m^{\beta}$, for $\beta$ purely imaginary, dates back the 1920's. Such functions are closely related to those used by Hardy and Littlewood \cite{HardyLittlewood1920} in their consideration of Waring's problem. Using the method of steepest descent, Burwell \cite{Burwell1924} deduced asymptotic expansions for $H_m^{\beta}$ for all $m>2$ (see also \cite{Bakhoom1933}). Concerning global bounds for $H_m^{\beta}$, we cannot expect to have estimates of the form \eqref{Hassymptoticnicecase} when $\beta$ is purely imaginary, for, in view of \eqref{heatker}, $x\mapsto |H_2^{\beta}(x)|$ is constant. When $m>2$, using oscillatory integral methods, we show that 
\begin{equation*}
|H_m^{\beta}(x)|\leq \frac{A}{|x|^{\frac{(m-2)}{2(m-1)}}}+\frac{B}{|x|}
\end{equation*}
for all real numbers $x$, where $A,B>0$. This estimate can also be deduced from the asymptotic expansions of Burwell \cite{Burwell1924}. Our estimates are seen to be sharp in view of this comparison.\\

\noindent Returning to our discussion of convolution powers, let us momentarily view the situation with a probabilistic eye. Suppose that $\mu$ is a signed Borel measure on $\mathbb{R}$ and $X_1,X_2,\dots$ are independent ``random'' variables each with distribution $\mu$. The distribution of the sum $S_n:=X_1+X_2+\cdots X_n$, for $n=1,2,\dots$, is the measure $\mu^{(n)}$ and can be computed by taking successive convolution powers of the measure $\mu$. Limit theorems are seen to be affirmative answers to the following question: Does $\mu^{(n)}$, properly scaled, converge in any sense as $n\rightarrow \infty$ and if so, to what? In \cite{Hochberg1978,Hochberg1980}, K. Hochberg proved a class of central limit theorems. They essentially state that, under certain conditions on $\mu$, there exists an even natural number $m\geq 2$ such that the signed Borel measures $\{\nu_n\}_{n\geq 1}$, defined by $\nu_n(B)=\mu^{(n)}(n^{1/m}B)$ for any Borel set $B$, converge weakly to the measure with density $H_m^1$ with 
respect to Lebesgue measure. R. Hersch \cite{RH} proved a class of central limit theorems in which ``random'' variables are allowed to take values in an abstract algebra over $\mathbb{R}$  (see also \cite{Zhukov1959}). Like Hochberg, Hersch's central limit theorems also involve weak convergence, however, the class of attractors in \cite{RH} is different. It consists of the Dirac mass and the measures with densities $H_m^{-i^m/m!}$ for all $m\geq 2$ such that $m\not\equiv 0\bmod 4$. Local limit theorems, by contrast, focus on convergence of the density of $\mu^{(n)}$. In our case, these are statements of uniform (or local uniform) convergence of $\phi^{(n)}(x)$ as $n\rightarrow\infty$. Local limit theorems, in the case that $\phi$ is generally real valued, were treated by I. Schoenberg \cite{IJS} and T. Greville \cite{TNEG} in connection to de Forest's problem in data smoothing. Their local limit theorems involve a certain subclass of our attractors, namely $H_m^{\beta}$ for $m\geq 2$ even and $\beta>0$, and 
$H_m^{i\tau}$ for $m>1$ odd and $\tau\in\mathbb{R}$. The local limit theorems of Schoenberg and Greville involve ad hoc assumptions that are too restrictive for us; Theorem \ref{mainstrong} extends their results. In the case that $\phi$ has admissible support, Theorem \ref{mainstrong} also extends the results of \cite{DSC1}. We refer the reader to Section $2$ of \cite{DSC1} for a brief review of local limit limit theorems and their connection to data smoothing and numerical difference schemes for partial differential equations.\\

\noindent The behavior of the convolution powers $\phi^{(n)}$ is determined by the local behavior of $\hat\phi$ by means of the Fourier inversion formula \eqref{fouriertransformidentity}. The latter two main results of this article, Theorems \ref{mainstrong} and \ref{mainweak}, are both stated under the assumption that $\sup_\xi|\hat\phi(\xi)|=1$; this can always be arranged by replacing $\phi$ by $A^{-1}\phi$ for an appropriate constant $A>0$. Theorems \ref{mainstrong} and \ref{mainweak} involve a number of constants and we now proceed to describe how they come about. First, we consider $\hat{\phi}(\xi)$ for $\xi\in(-\pi,\pi]$ and determine the set of points $\Omega\subseteq (-\pi,\pi]$ at which $|\hat\phi(\xi)|=\sup|\hat\phi|=1$. When $\phi$ is an aperiodic and irreducible random walk, this supremum is attained only at $0$ (see Lemma 2.3.1 of \cite{LL} and its subsequent remark), but in general, $|\hat\phi(\xi)|=1$ at multiple such points. In Section \ref{localphisec}, we show that the set $\Omega$ is 
finite. 
Second, for each $\xi_0\in\Omega$, we consider the Taylor expansion for $\log(\hat\phi(\xi+\xi_0)/\hat\phi(\xi_0))$ on a neighborhood of zero. In general, this series is of the form
\begin{equation*}
i\alpha\xi-\beta\xi^{m}+ o(\xi^{m})
\end{equation*}
as $\xi\rightarrow 0$, where $m=m(\xi_0)\in\{2,3,4,\dots\}$, $\alpha=\alpha(\xi_0)\in\mathbb{R}$ and $\beta=\beta(\xi_0)\in\mathbb{C}$ with $\Re(\beta(\xi_0))\geq 0$. Further, we show that $\Re(\beta(\xi_0))=0$ whenever $m(\xi_0)$ is odd. The constants $\alpha(\xi_0)$ and $\beta(\xi_0)$ play the roles of the mean and first non-vanishing moment of order $m(\xi_0)\geq 2$ for probability distributions (see Remark \ref{expectationremark} of Section \ref{localphisec}). Next, we set 
\begin{equation}\label{defofmphi}
m_\phi=\max_{\xi_0\in\Omega} m(\xi_0)
\end{equation}
and restrict our attention to the subset of points $\{\xi_1,\xi_2,\dots,\xi_R\}$ of $\Omega$ for which $m(\xi_q)=m_\phi$.  We show that the contribution to $\phi^{(n)}$ by $\hat\phi$ near $\xi_0\in\Omega(\phi)$ is on the order of $n^{-1/{m(\xi_0)}}$ (see Lemma \ref{bigolem}). Because $n^{-1/{m(\xi_0)}}=o(n^{-1/m_\phi})$ as $n\rightarrow\infty$ whenever $m(\xi_0)<m_\phi$, the influence on $\phi^{(n)}$ from such points is not seen in local limits; it is only the points $\xi_q$ for which $m(\xi_q)=m_\phi$ that matter. Finally, for each $q=1,2,\dots, R$, we set $\beta_q=\beta(\xi_q)$ and $\alpha_q=\alpha(\xi_q)$. \\

\noindent We now state our second main theorem, the first to involve local limits.

\begin{theorem}\label{mainstrong}
Let $\phi:\mathbb{Z}\rightarrow \mathbb{C}$ have admissible support and be such
that $\ \sup_\xi|\hat\phi(\xi)|=1$. Referring to the constants above and setting $m=m_\phi$, suppose additionally that 
\begin{equation}\label{mainstronghypoth}
 m>2\mbox{ or }\Re(\beta_q)>0\mbox{ for all }q=1,2,\dots, R.
\end{equation}
Then there exists a compact set $K\subseteq\mathbb{R}$ such that the supremum of $|\phi^{(n)}|$ is attained on 
\begin{equation}\label{mainstrongsetconclusion}
\left(\bigcup_{q=1}^{R}(\alpha_q n+ Kn^{1/m})\right)\bigcap \mathbb{Z}
\end{equation}
 and
\begin{equation}\label{mainstrongconvergenceconclusion}
\phi^{(n)}(x)=\sum_{q=1}^{R}n^{-1/m}e^{-ix\xi_q}\hat{\phi}(\xi_q)^n H_m^{\beta_q}\left(\frac{x-\alpha_qn}{n^{1/m}}\right)+o(n^{-1/m})
\end{equation}  
uniformly in $\mathbb{Z}$.  
\end{theorem}

\begin{remark}
If $\phi:\mathbb{Z}\rightarrow\mathbb{C}$ has admissible support and is such that $\sup_\xi|\hat\phi(\xi)|=1$, hypothesis \eqref{mainstronghypoth} is equivalent to the condition that, for every $\xi_0\in(-\pi,\pi]$ for which $|\hat{\phi}(\xi_0)|=1$,
\begin{equation*}
\frac{d^2}{d\xi^2}\log\hat{\phi}(\xi)\Big|_{\xi_0}\neq i\tau
\end{equation*}
for any non-zero real number $\tau$. 
\end{remark}

\noindent As the conclusion \eqref{mainstrongsetconclusion} suggests, the interesting behavior of $\phi^{(n)}$ occurs on the moving sets $\alpha_q n+ Kn^{1/m}$ called \emph{packets}. Each packet drifts with (and expands around) the point $\alpha_q n$ and so we call $\alpha_q$ a \emph{drift constant}.  There is much gained in studying $\phi^{(n)}$ by zooming in on its  packets, i.e., choosing a drift constant $\alpha_q$ from $\{\alpha_1,\alpha_2,\dots,\alpha_R\}$ and studying $\phi^{(n)}(\lfloor\alpha_qn+xn^{1/m}\rfloor)$ where $x$ lives in a compact set (see Subsection \ref{airyex}). In doing this, we arrive at our third main result, a local limit theorem in which only the points $\xi_l\in\{\xi_1,\xi_2,\dots,\xi_R\}$ and corresponding attractors $H_m^{\beta_l}$ appear, provided $\alpha_l=\alpha_q$.  

\begin{theorem}\label{mainweak}
Let $\phi:\mathbb{Z}\rightarrow\mathbb{C}$ have admissible support and be such
that $ \sup_\xi|\hat\phi(\xi)|=1$. Then, referring to the collections $\xi_1,\xi_2,\dots,\xi_R$, $\beta_1,\beta_2,\dots,\beta_R$ and $\alpha_1,\alpha_2,\dots,\alpha_R$ and setting $m=m_\phi$, the following holds: To each $\alpha_q$, there exist subcollections $\xi_{j_1},\xi_{j_2},\dots,\xi_{j_{r(q)}}$ and $\beta_{j_1},\beta_{j_2},\dots,\beta_{j_{r(q)}}$ such that
\begin{equation}\label{mainweakconvergenceconclusion}
 \phi^{(n)}(\lfloor\alpha_q n+xn^{1/m}\rfloor)=\sum_{j=1}^{r(q)}n^{-1/m}e^{-i\lfloor\alpha_q n+xn^{1/m}\rfloor\xi_{j_l}}\hat{\phi}(\xi_{j_l})^m H_m^{\beta_{j_l}}(x)+o(n^{-1/m})
\end{equation}
uniformly for $x$ in any compact set.
\end{theorem}

\noindent We note that Theorem \ref{mainweak} does not require the hypothesis \eqref{mainstronghypoth} of Theorem \ref{mainstrong}. The hypothesis rules out the situation in which $\phi^{(n)}$ is approximated by $H_2^{\beta}$ where $\beta$ is purely imaginary. Correspondingly, the example where $\phi$ is defined by \eqref{ex1def} can be treated by Theorem \ref{mainweak} but not Theorem \ref{mainstrong}.  For the generality gained by eliminating the hypothesis \eqref{mainstronghypoth} we lose the uniformity of the limit \eqref{mainstrongconvergenceconclusion} on all of $\mathbb{Z}$. As we illustrate in Subsection \ref{airyex}, the conclusion \eqref{mainweakconvergenceconclusion} is sometimes more informative anyway. The limits \eqref{mainstrongconvergenceconclusion} and \eqref{mainweakconvergenceconclusion} of Theorems \ref{mainstrong} and \ref{mainweak} both involve a sum of the attractors $H_m^{\beta}$. We remark that these sums are not identically zero and, in fact, each sum is bounded below in absolute 
value by $Cn^{-1/m}$ for some constant $C>0$; this is demonstrated in Section \ref{lowerboundsec} and is used to establish the lower estimate in Theorem \ref{mainbound}.\\

\noindent Everything in this article pertains to a single dimension. In a forthcoming article, we will study the convolution powers $\phi^{(n)}$ where $\phi:\mathbb{Z}^d\rightarrow \mathbb{C}$ is subject to some restrictive assumptions. Let us simply note here that the situation is more complicated in the $\mathbb{Z}^d$ setting. For example, it is not clear what the analogue of Theorem \ref{mainbound} should be. Further, at points $\xi_0\in(-\pi,\pi]^d$ where the Fourier transform satisfies $\sup_{\xi}|\hat\phi(\xi)|=|\hat\phi(\xi_0)|=1$, $|\hat\phi(\xi)|$ can decay at different rates along different directions. This behavior will be seen to affect local limits in which attractors exhibit anisotropic scaling. For instance, by taking a tensor product of admissible functions (in the sense of the present article), one can easily construct $\phi:\mathbb{Z}^2\rightarrow \mathbb{C}$ for which
\begin{equation*}
\phi^{(n)}(x)=n^{-3/4}H_2^{\beta_1}(n^{-1/2}x_1)H_4^{\beta_2}(n^{-1/4}x_2)+o(n^{-3/4})
\end{equation*}
uniformly for $x=(x_1,x_2)\in\mathbb{Z}^2$, where $\Re(\beta_1),\Re(\beta_2)>0$.\\

\noindent The paper is organized as follows: In Section \ref{localphisec}, we study the local
behavior of the Fourier transform of $\phi$. In Section \ref{upperboundsec}, we address some technical lemmas involving oscillatory integrals and prove the estimate $A^{-n}\|\phi^{(n)}\|_\infty\leq C'n^{-1/m}$ of Theorem \ref{mainbound}; this is Theorem \ref{firsthalfmainbound}. Section \ref{attractorsec} concentrates on the attractors $H_m^{\beta}$ where convergence, analyticity and global bounds are addressed. The local limit theorems of Theorems \ref{mainweak} and \ref{mainstrong}
are proven in Section \ref{locallimitsec}. In Section \ref{lowerboundsec}, we complete the proof of Theorem
\ref{mainbound} and in Section \ref{masssec}, the conclusion \eqref{mainstrongsetconclusion} of Theorem \ref{mainstrong}
is proven. Section \ref{exsec} gives examples and addresses a general situation previously treated in \cite{DSC1}.

\section{Local behavior of $\hat{\phi}$}\label{localphisec}

\noindent In this section we study the local behavior of $\hat{\phi}$ at points in $(-\pi,\pi]$
at which the supremum of $|\hat{\phi}|$ is attained. This will be seen to completely
determine the limiting behavior of the convolution powers of $\phi$. We proceed by making
some simple observations about \eqref{fouriertransform} under the assumptions
that $\phi$ has admissible support and $\sup_{\xi}|\hat{\phi}(\xi)|=1.$ Our first observation concerns the number of points at which $|\hat{\phi}(\xi)|=1.$ Because $\phi$ has admissible support, $|\hat{\phi}|^2$ is a non-constant trigonometric polynomial and so $|\hat{\phi}|$ is not constant. From here we observe that $\hat{\phi}$ can only satisfy $|\hat{\phi}(\xi)|=1$ at a finite number of points in $(-\pi,\pi]$; a simple accumulation-point argument shows the necessity of this fact. We now observe that $\hat{\phi}$ is a finite linear combination of exponentials and is
therefore analytic. We use this observation to study the local behavior of
$\hat{\phi}(\xi)$ about any point $\xi_0\in(-\pi,\pi]$ for which $|\hat{\phi}(\xi_0)|=1$. To
this end we consider 
\begin{equation}
\Gamma(\xi)=\log\left(\frac{\hat{\phi}(\xi+\xi_0)}{\hat{\phi}(\xi_0)}\right),
\end{equation}
where $\log$ is taken to be the principle branch of logarithm and we allow the
variable $\xi$ to be complex, for the time being. It follows from our 
remarks above that $\Gamma$ is analytic on an open neighborhood of $0$ and we can
therefore consider its convergent Taylor series
\begin{equation*}
 \Gamma(\xi)=\sum_{l=1}^\infty a_l \xi^l
\end{equation*} 
on this neighborhood. The limiting behavior of
the convolution powers of $\phi$ is characterized by the
first few non-zero terms of this series.\\

\noindent The requirement that $|\hat{\phi}(\xi)|\leq 1$ imposes conditions on
the Taylor expansion for $\Gamma$ as follows: We consider the collection
$\{a_l\}_{l=1}^{\infty}$ of coefficients of the series and let
$k=\min\{l\geq 1:\Re(a_l)\neq0\}$, which exists, for otherwise $|\hat{\phi}|$ would be
constant. We claim that $k$ is even and $\Re(a_k)<0$. To see this we observe that
by only considering real values of $\xi$ we can find a neighborhood of $0$ on
which
\begin{equation*}
e^{C\xi^k}\leq|\hat{\phi}(\xi+\xi_0)|=|\hat{\phi}(\xi_0)e^{\Gamma(\xi)}|
\end{equation*}
and where $C$ is a real constant having the same sign as $\Re(a_k)$. If it is
the case that $\Re(a_k)>0$ or $k$ is odd we have that $|\hat{\phi}(\xi+\xi_0)|>1$
for some $\xi$ which leads to a contradiction. We will summarize the above arguments shortly. First we give the following convenient definition, originally motivated by
Thom\'{e}e \cite{VT2}.

\begin{definition}\label{types}
Let  $\nu:\mathbb{R}\rightarrow\mathbb{C}$ be analytic on a neighborhood of a point $\xi_0$ for which $|\nu(\xi_0)|=1$. Let $\Gamma:\mathcal{O}\subseteq\mathbb{R}\rightarrow \mathbb{C}$ be defined by
\begin{equation}\label{Gammadef}
\Gamma(\xi)=\log\left(\frac{\nu(\xi+\xi_0)}{\nu(\xi_0)}\right),
\end{equation}
where $\mathcal{O}$ is an open neighborhood of $0$ and is such that $\mathcal{O}\ni\xi\mapsto\nu(\xi+\xi_0)$ is non-vanishing.

\noindent $1$. We say that $\xi_0$ is a point of type 1 and of order $m$ for
$\nu$ if the Taylor expansion for \eqref{Gammadef} yields an even integer $m\geq 2$, a real number $\alpha$, and a
complex number $\beta$ with $\Re(\beta)>0$ such that 
\begin{equation}\label{type1}
\Gamma(\xi)=i\alpha \xi -\beta \xi^m+\sum_{l=m+1}^{\infty}a_l\xi^l
\end{equation}
on $\mathcal{O}$. In this case we write $\xi_0\sim(1;m)$.\\

\noindent $2$. We say that $\xi_0$ is a point of type 2 and of order $m$ for
$\nu$ if the Taylor expansion for \eqref{Gammadef} yields $m,k, \alpha, \gamma, p(\xi)$, where $m$ and $k$ are
natural numbers with $k$ even and $1<m<k$; $\alpha$ and $\gamma$ are real
numbers with $\gamma>0$; and $p(\xi)$ is a real polynomial with $p(0)\neq 0$
such that 
\begin{equation}\label{type2}
\Gamma(\xi)=i\alpha \xi -i\xi^m p(\xi)-\gamma \xi^k+\sum_{l=k+1}^{\infty}a_l\xi^l.
\end{equation}
on $\mathcal{O}$. In this case we write $\xi_0\sim(2;m)$ and set $\beta=ip(0)$.\\

\noindent In both cases, the order $m$ refers to the degree of the first non-vanishing term, higher than degree one, in the Taylor expansion for $\Gamma$. The type refers to the complex phase of coefficient of this term: it is of type $1$ if the coefficient has a non-zero real part, otherwise it is of type $2$.   In either case we refer to the constant $\alpha$ as the drift
constant for $\xi_0$.
\end{definition}

\noindent Let us note that the neighborhood $\mathcal{O}$ in the above definition is immaterial; it needs only to be small enough to ensure that $\log$ is defined and analytic there. Using the definition, the arguments which preceded it are summarized in the following proposition.

\begin{proposition}\label{typesprop}
Let $\phi:\mathbb{Z}\rightarrow\mathbb{C}$ have admissible support.
Suppose that the Fourier transform of $\phi$
satisfies $\sup_{\xi}|\hat{\phi}(\xi)|=1$. Then 
\begin{equation*}
\Omega=\{\xi'\in(-\pi,\pi]:|\hat\phi(\xi')|=1\}
\end{equation*}
is finite and, for $\xi_0\in\Omega$, we have either
$\xi_0\sim(1;m)$ or $\xi_0\sim(2;m)$ for some natural number $m=m(\xi_0)\geq 2$.
\end{proposition}

\begin{convention}\label{constantsconvention}
For any $\phi:\mathbb{Z}\rightarrow\mathbb{C}$ satisfying the hypotheses of Proposition \ref{typesprop}, set
\begin{equation}\label{mdef}
m=\max_{\xi_0\in\Omega}m(\xi_0).
\end{equation}
In view of the proposition, we can write
\begin{equation*}
\Omega=\{\xi_1,\xi_2,\dots,\xi_Q\},
\end{equation*}
where we shall henceforth assume that $\Omega$ is indexed in the following way:
\begin{itemize}
\item  For each $q=1,2,\dots, R$, $\xi_q\sim(1;m_q)$ or $\xi_q\sim (2;m_q)$ with $m_q=m$ and associated constants $\alpha_q$ and $\beta_q$. 
\item  For each $q=R+1,R+2,\dots, Q$, $\xi_q\sim (1;m_q)$ or $\xi_q\sim (2;m_q)$ with $m_q<m$.
\end{itemize}
Hence, to the points $\{\xi_1,\xi_2,\dots,\xi_R\}\subseteq \Omega$ we have the associated collections $\alpha_1,\alpha_2,\dots,\alpha_R$ of real numbers and $\beta_1,\beta_2,\dots,\beta_R$ of non-zero complex numbers with $\Re(\beta_q)\geq 0$ for $q=1,2,\dots,R$.
\end{convention}
\noindent We remark that $\Omega$, $m\,(=m_\phi)$, and the constants $\alpha_q$ and $\beta_q$ for $q=1,2,\dots,R$ of Convention \ref{constantsconvention} are consistent with those of the discussion preceding the statement of Theorem \ref{mainstrong}. Therefore, the constants appearing in Theorems \ref{mainstrong} and \ref{mainweak} are those of Convention \ref{constantsconvention}.

\begin{remark}\label{expectationremark}
There is an alternative way to find the constants $m_q$ $\alpha_q$ and $\beta_q$ above. For any function $f:\mathbb{Z}\rightarrow \mathbb{C}$, define
\begin{equation*}
\mathbb{E}f(X)=\sum_{x\in\mathbb{Z}}f(x)\phi(x),
\end{equation*}
where $X$ is to be understood as a ``random'' variable with distribution $\phi$. For each $\xi_q\in\Omega$, put
\begin{equation*}
a(\xi_q)=\frac{\mathbb{E}Xe^{i\xi_qX}}{\hat\phi(\xi_q)}
\end{equation*}
and, for each natural number $k\geq 2$,
\begin{equation*}
b_k(\xi_q)=\frac{i^k}{k!}\left(a(\xi_q)^k-\frac{\mathbb{E} X^ke^{i\xi_qX}}{\hat\phi(\xi_q)}\right).
\end{equation*}
It is easily shown that $\alpha_q=a(\xi_q)$, $m_q=\min\{k\geq 2:b_k(\xi_q)\neq 0\}$  and $\beta_q=b_{m_q}(\xi_q)$. Proposition 2.4 of \cite{DSC1} gives a class of examples for which $\phi$ is real valued, $\Omega=\{0\}$ and $m=m_1=2l$ for any specified $l\geq 1$. Necessarily, $b_k(0)=0$ for all $k<2l$. 

If we further assume that $\phi\geq 0$ and $\Omega=\{0\}$, it follows that $b_2(0)\neq 0$ and so $m=2$. This situation is equivalent to the case in which $\phi$ is the distribution of a random variable $X$ with state space $\mathbb{Z}$ such that $\mbox{Supp}(\phi)$ is not contained in any proper subgroup of $\mathbb{Z}$. Here $\alpha_1=\mathbb{E} X$ and $2\beta_1=\mathbb{E}X^2-(\mathbb{E}X)^2=\mbox{Var}(X)$.  In this way, the standard local limit theorem is captured by Theorem \ref{mainstrong}. \\ 
\end{remark}

\noindent To exploit the interplay between local approximations of $\hat\phi$ and the Fourier inversion formula, it useful to consider a domain of integration $T$ in which $\Omega$ sits in the interior. To this end, let $\eta\geq 0$ be such that $\Omega\subseteq (-\pi+\eta,\pi+\eta)$ and set $T=(-\pi+\eta,\pi+\eta]$. Of course, for each natural number $n$ and $x\in\mathbb{Z}$, we have
\begin{equation}\label{fouriertransformidentity}
\phi^{(n)}(x)=\frac{1}{2\pi}\int_{T}e^{-ix\xi}\hat\phi(\xi)\,d\xi.
\end{equation}
It is also useful to consider the following extension of $\phi^{n}(x)$: Define $\phi_{\mbox{\scriptsize{e}}}:\mathbb{N}\times\mathbb{R}\rightarrow\mathbb{C}$ by
\begin{equation}\label{phiextdef}
\phi_{\mbox{\scriptsize{e}}}(n,x)=\frac{1}{2\pi}\int_{T}e^{-ix\xi}\hat\phi(\xi)\,d\xi
\end{equation}
for $n\in\mathbb{N}$ and $x\in\mathbb{R}$. We note that $\phi_{\mbox{\scriptsize{e}}}(n,x)=\phi^{(n)}(x)$ for all $n\in\mathbb{N}$ and $x\in\mathbb{Z}$.\\

\noindent The following lemma is seen to govern the limiting behavior of the convolution powers of $\phi$.

\begin{lemma}\label{typeslemma}
Let $\nu:\mathbb{R}\rightarrow\mathbb{C}$ be analytic on a neighborhood of a point $\xi_0$ such that $|\nu(\xi_0)|=1$.

\noindent $1.$ If $\xi_0\sim(1;m)$, then there exist $\delta>0$ and $B,C>0$ such that
\begin{equation}\label{type1approx}
|\Gamma(\xi)-i\alpha\xi+\beta\xi^m|\leq B|\xi|^{m+1}
\end{equation}
and
\begin{equation}\label{type1bound}
|\nu(\xi+\xi_0)|\leq e^{-C\xi^m}
\end{equation}
for all $|\xi|\leq\delta$. Here $\Gamma$, $\alpha$, and $\beta$ are given by Definition \ref{types}.\\

\noindent $2.$ If $\xi_0\sim(2;m)$, there exists $\delta>0$ and $B>0$ such that
\begin{equation}\label{type2approx}
|\Gamma(\xi)-i\alpha\xi+ip(\xi)\xi^m|\leq B\xi^{k}
\end{equation}
for all $|\xi|\leq \delta$. Moreover, there exist $C,D>0$ such that the function
\begin{equation*}
 g(\xi)=\nu(\xi_0)^{-1}\nu
(\xi+\xi_0)\exp(-i\alpha\xi+i\xi^mp(\xi))
\end{equation*}
satisfies
\begin{equation}\label{type2bound}
 |g(\xi)|\leq e^{-C\xi^k}
\end{equation}
and 
\begin{equation}\label{type2derbound}
|g'(\xi)|\leq D|\xi|^{k-1}e^{-C\xi^k}
\end{equation}
for all $|\xi|\leq \delta$. Here $\Gamma, k, \alpha$, and $p(\xi)$ are given by Definition \ref{types}.
\end{lemma}

\begin{proof}
By our definitions, we have
\begin{equation*}
\nu(\xi+\xi_0)=\nu(\xi_0)e^{\Gamma(\xi)},
\end{equation*}
where $\Gamma$ is defined by \eqref{Gammadef}.
In the case that $\xi_0\sim(1;m)$, \eqref{type1approx} and \eqref{type1bound} are immediate from \eqref{type1} and the fact that the series $\sum_{l=m+1}a_l\xi^l$ converges uniformly on a neighborhood of $0$.

In the case that $\xi_0\sim(2;m)$, the justification of the estimates \eqref{type2approx} and \eqref{type2bound} follows similarly. For the last conclusion, we observe that
\begin{eqnarray*}
 g'(\xi)&=&\frac{d}{d\xi}\exp(-i\alpha\xi+i\xi^m p(\xi))e^{\Gamma(\xi)}\\
&=&\frac{d}{d\xi}\exp{\left(-\gamma\xi^k+\sum_{l=k+1}^{\infty}a_l\xi^l\right)}\\
&=&\left(-\gamma k\xi^{k-1}+\sum_{l=k+1}^{\infty}a_ll\xi^{l-1}\right)g(\xi)
\end{eqnarray*}
on a neighborhood of $0$. The inequality \eqref{type2derbound} now follows without trouble.
\end{proof}

\section{The upper bound}\label{upperboundsec}

\noindent The goal of this section is to establish the upper bound of Theorem \ref{mainbound}. To this end,
we address a series of technical lemmas involving oscillatory integrals of the
form 
\begin{equation*}
\int_a^b g(\xi)e^{if(\xi)}d\xi
\end{equation*}
which are used throughout the remainder of the paper. Many of the
arguments within are based on the same or slightly less general arguments made
by Greville \cite{TNEG}, Thom\'{e}e \cite{VT2} and, not surprisingly, van der
Corput. 

\begin{lemma}\label{boundbyparts}
 Let $h\in L^1([a,b])$ and $g\in\mathcal{C}^1([a,b])$ be complex valued. Then
for any $M$ such that
\begin{equation*}
\left|\int_a^x h(u)du\right|\leq M
\end{equation*}
for all $x\in[a,b]$ we have
\begin{equation*}
\left|\int_a^b g(u)h(u)du\right|\leq M\left(\|g\|_\infty+\|g'\|_1\right).
\end{equation*}
\end{lemma}
\begin{proof}
 For $h\in L^1([a,b])$, the function
\begin{equation*}
 f(x)=\int_b^xh(u)du
\end{equation*}
is absolutely continuous and $f'(x)=h(x)$ almost everywhere. Furthermore, our
hypothesis guarantees that $|f(x)|\leq M$ for all $x\in[a,b]$. Integration by
parts yields
\begin{equation*}
\int_a^b g(u)h(u)du=\left[g(u)f(u)\right]_a^b-\int_a^b g'(u)f(u)du
\end{equation*}
and therefore
\begin{eqnarray*}
 \left|\int_a^b g(u)h(u)du\right|&\leq&|f(b)g(b)|+0+\int_a^b|f(u)\|g'(u)|du\\
&\leq& M\|g\|_{\infty}+M\|g'\|_1.
\end{eqnarray*}
\end{proof}

\noindent The following two lemmas, \ref{fprimebound} and \ref{fdoubleprimebound}, are due originally to van der Corput. The proof of Lemma \ref{fprimebound} is a nice application of the second mean value theorem for integrals and can be found in \cite{VT2}. We note that Lemma $2.3$ of \cite{VT2} is stated under slightly stronger hypotheses than Lemma \ref{fprimebound}, however the proof yields our statement exactly. The validity of Lemma \ref{fprimebound} can also be seen using alternating series \cite{TNEG}. For the proof of Lemma \ref{fdoubleprimebound}, we refer the reader to Lemma 3.3 of \cite{VT2}; its proof is relatively simple but involves checking several cases (see also Chapter 1 of \cite{BTW}).

\begin{lemma}\label{fprimebound}
Let $f\in\mathcal{C}^1([a,b])$ be real valued and suppose that $f'$ is a monotonic function such that $f'(x)\neq0$ for all $x\in[a,b]$. 
Then,
\begin{equation}
\left|\int_a^b e^{if(u)}du\right|\leq \frac{4}{\lambda},
\end{equation}
where
\begin{equation}
\lambda=\inf_{x\in [a,b]}|f'(x)|.
\end{equation} 
\end{lemma}

\begin{lemma}\label{fdoubleprimebound}
Let $f\in\mathcal{C}^2([a,b])$ be real valued and suppose that $f''(x)\neq 0$ for all $x\in[a,b]$. Then
\begin{equation*}
\left|\int_a^b e^{if(u)}du\right|\leq \frac{8}{\sqrt{\rho}},
\end{equation*}
where
\begin{equation*}
 \rho=\inf_{x\in[a,b]}|f''(x)|.
\end{equation*}
\end{lemma}

\begin{lemma}\label{minbound}
Let $g\in\mathcal{C}^1([a,b])$ be
complex valued and let $f\in \mathcal{C}^2([a,b])$ be real valued and such that $f''(x)\neq 0$ for all $x\in[a,b]$. Then
\begin{equation*}
\left|\int_a^bg(u)e^{if(u)}du\right|\leq
\min\left\{\frac{4}{\lambda},\frac{8}{\sqrt{\rho}}\right\}\left(\|g\|_{\infty}
+\|g'\|_1\right),
\end{equation*}
where $\lambda=\inf_{x\in[a,b]}|f'(x)|$ and $\rho=\inf_{x\in[a,b]}|f''(x)|$.
\end{lemma}
\begin{proof}
Combining the results of Lemmas \ref{fprimebound} and \ref{fdoubleprimebound} show
\begin{equation*}
\left|\int_a^x e^{if(u)}du\right|\leq \min\left\{\frac{4}{\lambda},\frac{8}{\sqrt{\rho}}\right\}
\end{equation*} for any $x\in[a,b]$. We remark that $4/\lambda$ only
contributes to the upper bound provided $f'$ is never zero, in which case the
application of Lemma \ref{fprimebound} is justified. Setting $h(u)=e^{if(u)}$ we
note that the functions $g$ and $h$ are the subject of Lemma \ref{boundbyparts}.
The result now follows immediately from Lemma \ref{boundbyparts}.
\end{proof}

\begin{lemma}\label{bigolem}
Let $\nu:\mathbb{R}\rightarrow\mathbb{C}$ be analytic on a neighborhood of $\xi_0$ where $|\nu(\xi_0)|=1$. If $\xi_0$ is a point of order $m\geq 2$ for $\nu$, then there is
$\delta>0$ such that
\begin{equation*}
\frac{1}{2\pi}\int_{|\xi-\xi_0|\leq \delta}\nu(\xi)^n
e^{-ix\xi}d\xi=O(n^{-1/m}),
\end{equation*}
where the limit is uniform in $x\in \mathbb{R}$.
\end{lemma}
\begin{proof}
 Let us first assume that $\xi_0\sim(1;m)$. Our hypothesis guarantees that $m$ is
even and by Lemma \ref{typeslemma} there are constants $C>0$ and $\delta>0$ such that
\begin{equation*}
|\nu(\xi+\xi_0)|\leq e^{-C\xi^m}
\end{equation*}
for all $-\delta\leq\xi\leq \delta$. Therefore
\begin{eqnarray*}
\left| \frac{1}{2\pi}\int_{|\xi-\xi_0|\leq \delta}\nu(\xi)^n
e^{-ix\xi}d\xi\right|&\leq& \int_{-\delta}^\delta |\nu(\xi+\xi_0)|^nd\xi\\
&\leq&\int_{\mathbb{R}}e^{-nC\xi^m}d\xi\\
&\leq& \frac{M}{n^{1/m}}.
\end{eqnarray*}

In the second case we assume that $\xi_0\sim(2;m)$. We set
\begin{equation*}
g(\xi)=\left[\nu(\xi_0)^{-1}\nu(\xi+\xi_0)\exp(-i\alpha\xi+i\xi^mp(\xi))\right
]
\end{equation*}

and
\begin{equation*}
 f_n(\xi,x)=(n\alpha-x) \xi-n\xi^mp(\xi).
\end{equation*}
We note that $f_n$ is real valued. Appealing to Lemma \ref{typeslemma}, let $\delta>0$ be chosen so that the estimates \eqref{type2bound} and \eqref{type2derbound} hold for all $\xi\in[-\delta,\delta]$. Upon changing variables of integration and using the fact that $|\nu(\xi_0)|=1$, we write
\begin{equation*}
 \left|\frac{1}{2\pi}\int_{|\xi-\xi_0|\leq \delta}\nu(\xi)^n
e^{-ix\xi}d\xi\right|\leq\sum_{j=1}^3\left|\int_{I_j}g(\xi)^ne^{if_n(\xi,x)}d\xi\right|,
\end{equation*}
where $I_1=[-\delta,-n^{1/m}]$, $I_2=(-n^{1/m},n^{1/m})$ and $I_3=[n^{1/m},\delta]$. On the interval $I_2$, $|g(\xi)|\leq 1$ by \eqref{type2bound} and therefore
\begin{equation*}
 \left|\int_{I_2}g(\xi)^ne^{if_n(\xi,x)}d\xi\right|\leq \frac{2}{n^{1/m}}.
\end{equation*}

We now consider the integral over $I_1$ to which we will apply Lemma \ref{minbound}. First observe that the regularity requirements of Lemma \ref{minbound} for $f_n$ and $g^n$ are met. Differentiating $f_n$ twice with respect to $\xi$ gives
\begin{equation*}
 \partial_{\xi}^2 f_n(\xi,x)=-n\frac{d^2}{d\xi^2}\xi^mp(\xi),
\end{equation*}
which is independent of $x$. Using the fact that $\xi^mp(\xi)$ is a polynomial with
$m$ being the smallest power of its terms, we may further restrict $\delta>0$ so
that
\begin{equation}\label{polylowerbound}
 C^2|\xi|^{m-2}\leq \left|\frac{d^2}{d\xi^2}\xi^mp(\xi)\right|
\end{equation}
for some $C>0$ and for all $\xi\in[-\delta,\delta]$. Consequently $|\partial_{\xi}^2 f_n(\xi,x)|> 0$ for all $\xi\in I_1$ and $x\in\mathbb{R}$. Appealing to Lemma \ref{fdoubleprimebound} we have
\begin{equation}\label{bigoint1}
 \left|\int_{I_1}g(\xi)^ne^{if_n(\xi,x)}d\xi\right|\leq\frac{8}{\sqrt{\lambda}}(\|g^n\|_{\infty}+\|ng'g^{n-1}\|_1),
\end{equation}
where $\lambda=\inf_{\xi\in I_1}|\partial_{\xi}^2 f_n(\xi,x)|$. Using \eqref{polylowerbound} and recalling that $m\geq 2$ we observe that
\begin{equation*}
Cn^{1/m}=\sqrt{C^2n|-n^{-1/m}|^{m-2}}\leq\sqrt{\inf_{I_1}
C^2n|\xi|^{m-2}}\leq \sqrt{\lambda}.
\end{equation*}
Now by \eqref{type2bound} and \eqref{type2derbound} of Lemma \ref{typeslemma} we have $\|g^n\|_{\infty}\leq 1$ and
\begin{eqnarray*}
 \|ng'g^{n-1}\|_1&=&n\int_{I_1}|g'(\xi)g(\xi)^{n-1}|d\xi\\
&\leq& n\int_{I_1}D|\xi|^{k-1}e^{-nC\xi^k}d\xi\\
&\leq&\int_{\mathbb{R}}D|u|^{k-1}e^{-Cu^k}du=M<\infty.
\end{eqnarray*}
Inserting the above estimates into \eqref{bigoint1} gives
\begin{equation*}
 \left|\int_{I_1}g(\xi)^ne^{if_n(\xi,x)}d\xi\right|\leq\frac{8(1+M)}{Cn^{1/m}}=\frac{K_1}{n^{1/m}}.
\end{equation*}
A similar calculation shows that 
\begin{equation*}
\left|\int_{I_3}g_n(\xi)e^{if_n(\xi,x)}d\xi\right|\leq\frac{K_2}{n^{1/m}}
\end{equation*}
for some $K_2>0$. Putting these estimates together gives
\begin{equation*}
  \left|\frac{1}{2\pi}\int_{|\xi-\xi_0|\leq \delta}\nu(\xi)^n
e^{-ix\xi}d\xi\right|\leq
\frac{K_1}{n^{1/m}}+\frac{2}{n^{1/m}}+\frac{K_2}{n^{1/m}},
\end{equation*}
our desired inequality.
\end{proof}

\begin{theorem}\label{firsthalfmainbound}
Let $\phi:\mathbb{Z}\rightarrow \mathbb{C}$ have admissible support and let $A=\sup_\xi|\hat{\phi}(\xi)|$. Then there is a
natural number $m\geq 2$ and a real number $C'>0$ such that 
\begin{equation}\label{upperboundsup}
 A^{-n}\|\phi^{(n)}\|_{\infty}\leq C' n^{-1/m}
\end{equation}
for all natural numbers $n$.
\end{theorem}

\begin{proof}
It suffices to prove the theorem in the case that $A=\sup_{\xi}|\hat{\phi}(\xi)|=1$, for otherwise one simply
multiplies the Fourier inversion formula by $A^{-n}$. In view of Proposition \ref{typesprop}, we adopt Convention \ref{constantsconvention}. For each $\xi_q\in\Omega$ with associated $2\leq m_q\leq m$, select $\delta_q>0$ for which the conclusion of Lemma \ref{bigolem} holds and small enough to ensure that the intervals $I_q:=[\xi_q-\delta_q,\xi_q+\delta_p]\subseteq T$ for $i=1,2,\dots,Q$ are disjoint. Set $J=T\setminus \cup_q I_q$ and $s=\sup_{\xi\in J}|\hat{\phi}(\xi)|<1$. Using \eqref{fouriertransformidentity}, we write
\begin{eqnarray*}
|\phi^{(n)}(x)|&=& \left|\sum_{q=1}^Q \frac{1}{2\pi}\int_{I_q} \hat{\phi}(\xi)^n e^{-ix\xi}d\xi+ \frac{1}{2\pi}\int_J \hat{\phi}(\xi)^n e^{-ix\xi}d\xi\right|\\
&\leq&\sum_{q=1}^Q \left|\frac{1}{2\pi}\int_{I_q} \hat{\phi}(\xi)^n e^{-ix\xi}d\xi\right|+ \frac{1}{2\pi}\int_J | \hat{\phi}(\xi)|^n d\xi\\
&\leq&\sum_{q=1}^Q \left|\frac{1}{2\pi}\int_{I_q} \hat{\phi}(\xi)^n e^{-ix\xi}d\xi\right|+ s^n.\\
\end{eqnarray*}
Using Lemma \ref{bigolem} we conclude that for every $x\in\mathbb{R}$
\begin{eqnarray*}
|\phi^{(n)}(x)|&\leq&\sum_{q=1}^Q \frac{K_q}{n^{1/m_q}}+s^n\\
&\leq& \frac{K}{n^{1/m}}+s^n
\end{eqnarray*} 
from which the result follows.
\end{proof}

\section{The attractors $H_m^{\beta}$}\label{attractorsec}

In this section we study the functions $H_m^{\beta}$ defined by \eqref{Hdef}. Our first task is to show that the integral defining $H_m^{\beta}$ converges in the senses indicated in the introduction.

\begin{proposition}\label{Hconverge} Let $m\geq 2$ be a natural number and let $\beta$ be a non-zero complex number such that $\Re(\beta)\geq 0$.
\begin{enumerate}
\item If $m$ is even and $\Re{\beta}>0$ then the integral defining $H_m^{\beta}(x)$ in \eqref{Hdef} converges absolutely and uniformly in $x$ on $\mathbb{R}$ as an improper Riemann integral.
\item If $m>2$ and $\Re(\beta)=0$ then the integral defining $H_m^{\beta}(x)$ converges uniformly in $x$ on $\mathbb{R}$ as an improper Riemann integral.
\item If $m=2$ and $\Re(\beta)=0$ then for any compact set $K\subseteq\mathbb{R}$, the integral defining $H_m^{\beta}(x)$ converges uniformly in $x$ on $K$ as an improper Riemann integral.
\end{enumerate}
\end{proposition}
\begin{proof}

For item $1$ there is nothing to prove, the result follows from the classical theory of Fourier transforms. For items $\mathit{2}$ and $\mathit{3}$, let $\tau$ be the non-zero real number such that $\beta=i\tau$ and set $f(u,x)=-xu-\tau u^m$. Our job is to show that the integral
\begin{equation*}
 \int_{\mathbb{R}}e^{if(x,u)}du
\end{equation*}
converges in the senses indicated for $m>2$ and $m=2$ respectively. 

We first consider the case where $m>2$. Let $\epsilon>0$ and choose $M$ sufficiently large so that
\begin{equation}\label{Mepsilonrelation}
 \frac{8}{\sqrt{|\tau|m(m-1)M^{m-2}}}\leq\epsilon.
\end{equation}
Observe that for any real numbers $a$ and $b$ such that $M< a\leq b$ or $a\leq b< -M$,
\begin{equation*}
|\tau|m(m-1)M^{m-2}< \inf_{u\in[a,b]}|\partial_{u}^2 f(x,u)|=:\lambda.
\end{equation*} We now apply Lemma \ref{fdoubleprimebound} and conclude that
\begin{equation*}
\left|\int_a^be^{-ixu-\beta u^m}du\right|=\left|\int_a^be^{if(x,u)}du\right|\leq\frac{8}{\sqrt{\lambda}}<\epsilon
\end{equation*}
for all $x\in\mathbb{R}$ and for all $a\leq b$ such that the distance from the interval $[a,b]$ to $0$ is more than $M$.  The Cauchy criterion for uniform convergence guarantees that the improper Riemann integrals
\begin{equation*}
 \int_0^{\infty}e^{-ixu-\beta u^m}du\hspace{.2cm}\text{and}\hspace{.2cm}\int_{-\infty}^0e^{-ixu-\beta u^m}du
\end{equation*}
converge uniformly in $x$ on $\mathbb{R}$. This proves item $\mathit{2}$.

Let us now assume that $m=2$. We remark that the above argument fails in this case because $\partial_{u}^2 f$ is a non-zero constant. Consequently, we need to use $\partial_{u} f$, which depends on both $u$ and $x$, to bound our integrals. Let $\epsilon>0$ and let $K\subseteq \mathbb{R}$ be a compact set. We choose $M>0$ so that 
\begin{equation*}
 \frac{4}{|2\tau M+x|}<\epsilon
\end{equation*}
for all $x\in K$. By applying Lemma \ref{fprimebound} and making an argument analogous to that given in the previous case we conclude that
\begin{equation*}
\left|\int_a^be^{-ixu-\beta u^2}du\right|=\left|\int_a^be^{if(x,u)}du\right|<\epsilon
\end{equation*}
for all $x\in K$ and for all $a\leq b$ such that the distance from the interval $[a,b]$ to $0$ is more than $M$.  Again, an application of the Cauchy criterion gives the desired result.
\end{proof}

\begin{proposition}\label{Hassymptoticprop}
Let $\beta$ be non-zero and purely imaginary. Then for any natural number $m>2$ there exist positive constants $A,B$ such that
\begin{equation}\label{Hassymptotic}
|H_m^{\beta}(x)|\leq\frac{A}{|x|^{\frac{m-2}{2(m-1)}}}+\frac{B}{|x|}
\end{equation}
for all $x\in\mathbb{R}$.
\end{proposition}
\begin{proof}
Let $\beta=i\tau$, where $\tau$ is a non-zero real number, and set \break$f(u,x)=-xu-\tau u^m$. For $x\neq 0$, put $M=(2m|\tau|/|x|)^{-1/(m-1)}$ and write
\begin{equation}\label{Hboundpropeq1}
H_m^{\beta}(x)=\frac{1}{2\pi}\int_{-\infty}^{-M}e^{if(u,x)}\,du+\frac{1}{2\pi}\int_{M}^{\infty}e^{if(u,x)}\,du+\frac{1}{2\pi}\int_{-M}^{M}e^{if(u,x)}\,du.
\end{equation}
Observe that
\begin{eqnarray*}
\inf_{u\in[-M,M]}|\partial_uf(u,x)|&=&\inf_{u\in[-M,M]}|x+m\tau u^{m-1}|\\
&=& m|\tau|\inf_{u\in[-M,M]}\left|\frac{x}{m\tau}+u^{m-1}\right|\geq m|\tau| M^{m-1}=\frac{|x|}{2},
\end{eqnarray*}
and therefore
\begin{equation}\label{Hboundpropeq2}
\left|\int_{-M}^{M}e^{if(u,x)}\,du\right|\leq\frac{8}{|x|} 
\end{equation}
in view of Lemma \ref{fprimebound}. Similarly, there exists $C>0$ such that for any $N>M$,
\begin{equation*}
\inf_{u\in[M,N]}|\partial_u^2f(u,x)|\geq \frac{|x|^{\frac{m-2}{m-1}}}{C^2}.
\end{equation*}
Thus, by appealing to Lemma \ref{fdoubleprimebound} and Proposition \ref{Hconverge}, we have
\begin{eqnarray*}
\left|\int_{M}^{\infty}e^{if(u,x)}\,du\right|&=&\lim_{N\rightarrow\infty}\left|\int_M^N e^{if(u,x)}\,du\right|\\
&\leq&\limsup_N \frac{8}{\sqrt{\inf_{u\in[M,N]}|\partial_u^2f(u,x)|}}\leq\frac{C}{|x|^{\frac{m-2}{2(m-1)}}}.
\end{eqnarray*}
By an analogous computation,
\begin{equation}\label{Hboundpropeq3}
\left|\int_{-\infty}^{M}e^{if(u,x)}\,du\right|\leq \frac{C}{|x|^{\frac{m-2}{2(m-1)}}}
\end{equation}
for $C>0$. The desired result follows by combining the estimates \eqref{Hboundpropeq1}, \eqref{Hboundpropeq2} and \eqref{Hboundpropeq3}.
\end{proof}

\noindent The final proposition of this section, Proposition \ref{Hanalytic}, asserts the analyticity and non-triviality of the functions $H_m^{\beta}$ for all values of $m$ and $\beta$ considered above. To preface it, let's consider the case in which $m\geq 2$ is even and $\Re(\beta)>0$: For any $x\in\mathbb{R}$, observe that
\begin{equation*}
H_m^{\beta}(x)=\frac{1}{2\pi}\int_{\mathbb{R}}e^{-ixu}e^{-\beta u^m}\,du=\frac{1}{2\pi}\int_{\mathbb{R}}\sum_{k=0}^{\infty}\frac{(-ixu)^k}{k!}e^{-\beta u^m}\,du.
\end{equation*}
Setting $2b=\Re(\beta)$, note that
\begin{equation*}
\int_{\mathbb{R}}\sum_{k=0}^{\infty}\left|\frac{(-ixu)^k}{k!}e^{-\beta u^m}\right|\,du= \int_{\mathbb{R}}\sum_{k=0}^{\infty}\frac{|xu|^k}{k!}e^{-2b u^m}\,du=\int_{\mathbb{R}}e^{|xu|-bu^m}e^{-bu^m}\,du.
\end{equation*}
By a simple maximization argument, one finds that 
$|xu|-bu^m\leq c |x|^{m/(m-1)}$ for all $u\in\mathbb{R}$, where $c=(1-m^{-1})(mb)^{-1/(m-1)}>0$. Therefore,
\begin{equation*}
\int_{\mathbb{R}}\sum_{k=0}^{\infty}\left|\frac{(-ixu)^k}{k!}e^{-\beta u^m}\right|\,du\leq e^{c|x|^{m/(m-1)}}\int_{\mathbb{R}}e^{-bu^m}\,du<\infty
\end{equation*}
and so this application of Tonelli's theorem justifies the following use of Fubini's theorem:
\begin{equation*}
\int_{\mathbb{R}}\sum_{k=0}^{\infty}\frac{(-ixu)^k}{k!}e^{-\beta u^m}\,du=\sum_{k=0}^{\infty}\int_{\mathbb{R}}\frac{(-iu)^kx^k}{k!}e^{-\beta u^m}\,du.
\end{equation*}
Therefore 
\begin{equation}\label{Hanalyticeq1}
H_m^{\beta}(x)=\frac{1}{2\pi}\sum_{k=0}^{\infty}\left(\int_{\mathbb{R}}\frac{(-iu)^k}{k!}e^{-\beta u^m}\,du\right)x^k
\end{equation}
for each $x\in\mathbb{R}$; note that the convergence of the series is part of the conclusion. Consequently, $H_m^{\beta}$ is analytic on $\mathbb{R}$. Moreover, from the representation \eqref{Hanalyticeq1}, it is clear that $H_m^{\beta}(x)$ is not identically zero. When $m>1$ is odd and $\beta$ is purely imaginary, the same conclusion was reached by R. Hersch \cite{RH}. His proof involves changing the contour of integration from $\mathbb{R}$ to a pair of rays on which the integrand is absolutely integrable. When $m\geq 2$ is even and $\beta$ is purely imaginary, Hersh's argument pushes through with very little modification.  We therefore summarize the result below and, in the case that $\Re(\beta)=0$, refer the reader to Theorem 4 of \cite{RH} for the essential details.

\begin{proposition}\label{Hanalytic}
Let $m\geq 2$ be a natural number and let $\beta$ be a non-zero complex number
with $\Re(\beta)\geq 0$. If $\Re(\beta)>0$ additionally assume that
$m$ is even. Then $H_m^{\beta}$ is analytic and not identically zero.
\end{proposition}

\section{Local limits}\label{locallimitsec}

\noindent In this section we prove Theorem \ref{mainweak} and the second conclusion, \eqref{mainstrongconvergenceconclusion}, of Theorem \ref{mainstrong}. To this end, the following three lemmas, Lemmas \ref{easylemma}, \ref{hardlemma22} and \ref{hardlemma2m}, focus on local approximations to Fourier-type integrals involving integer powers of an analytic function $\nu$ near a point $\xi_0$ at which $|\nu(\xi_0)|=1$. The lemmas treat the cases in which $\xi_0\sim(1;m)$, $\xi_0\sim(2;2)$ and $\xi_0\sim(2;m)$, respectively. The approximants are precisely the functions $H_m^{\beta}$ studied in the previous section.

\begin{lemma}\label{easylemma}
Let  $\nu:\mathbb{R}\rightarrow\mathbb{C}$ be analytic on a neighborhood of a point $\xi_0$ for which $|\nu(\xi_0)|=1$. Assume that $\xi_0\sim (1;m)$ with associated constants $\alpha$ and
$\beta$. Then for all $\epsilon>0$ there is a $\delta >0$
and a natural number $N$ such that
\begin{equation}
 \left|\frac{n^{1/m}}{2\pi}\int_{|\xi-\xi_0|\leq\delta}\nu(\xi)^n
e^{-ix\xi}d\xi-e^{-ix\xi_0}\nu(\xi_0)^n H_{m}^{\beta}\left(\frac{x-\alpha
n}{n^{1/m}}\right)\right|<\epsilon
\end{equation}
for all $n>N$ and for all real numbers $x$.
\end{lemma}

\begin{proof}
Let $\epsilon>0$ and set

\begin{equation*}
 y_n=(x-\alpha n)n^{-1/m}
\end{equation*}
and
\begin{equation*}
 g(u)=\left[\nu(\xi_0)^{-1} e^{-i\alpha u n^{-1/m}}\nu(\xi_0+u
n^{-1/m})\right]. 
\end{equation*}
Upon changing variables of integration we have
\begin{eqnarray*}
\lefteqn{\frac{n^{1/m}}{2\pi}\int_{|\xi-\xi_0|\leq\delta}\nu(\xi)^n e^{-ix\xi}d\xi}\\
&=&\frac{e^{-ix\xi_0}\nu(\xi_0)^n}{2\pi }\int_{|u|\leq\delta
n^{1/m}}\left[\nu(\xi_0)^{-1}e^{-i\alpha u n^{-1/m}}\nu(\xi_0+u
n^{-1/m})\right]^n e^{-iu\frac{x-\alpha n}{n^{1/m}}}du\\
&=&\frac{e^{-ix\xi_0}\nu(\xi_0)^n}{2\pi}\int_{|u|\leq\delta
n^{1/m}}g(u)^ne^{-iuy_n}du.
\end{eqnarray*}
Comparing the above integral with $e^{-ix\xi_0}\nu(\xi_0)^n H_{m}^{\beta}(y_n)$
gives
\begin{eqnarray*}
 \lefteqn{\left|\frac{n^{1/m}}{2\pi}\int_{|\xi-\xi_0|\leq\delta}\nu(\xi)^n e^{-ix\xi}d\xi-e^{-ix\xi_0}\nu(\xi_0)^n H_{m}^{\beta}(y_n)\right|}\\
&\leq&\left|\frac{e^{-ix\xi_0}\nu(\xi_0)^n}{2\pi }\int_{|u|\leq M}[g(u)^n-e^{-\beta u^m}]e^{-iuy_n}du\right|\\
&&+\left|\frac{e^{-ix\xi_0}\nu(\xi_0)^n}{2\pi }\int_{M<|u|\leq\delta
n^{1/m}}g(u)^ne^{-iuy_n}du\right|\\
&&+\left|\frac{e^{-ix\xi_0}\nu(\xi_0)^n}{2\pi }\int_{|u|>M}e^{-\beta
u^m}e^{-iuy_n}du\right|\\
&\leq&\int_{-M}^M|g(u)^n-e^{-\beta
u^m}|du+\int_{M<|u|\leq \delta n^{1/m}}|g(u)|^n du+\int_{|u|>M}e^{-\Re(\beta) u^m}du\\
&&=:\mathcal{I}_1+\mathcal{I}_2+\mathcal{I}_3,
\end{eqnarray*}
where $M<\delta n^{1/m}$ will soon be fixed. Notice that $\mathcal{I}_1$, $\mathcal{I}_2$ and $\mathcal{I}_3$ are independent of $x$.

In view of Lemma \ref{typeslemma}, there is $\delta>0$ and $C>0$ for which
\begin{equation}\label{easylemmaeq1}
|g(u)|^n=|v(\xi_0+un^{-1/m})|^n\leq (e^{-C(un^{-1/m})^m})^n=e^{-Cu^m}
\end{equation}
whenever $|u|\leq\delta n^{1/m}$. Therefore, 
\begin{equation*}
\mathcal{I}_2\leq \int_{M<|u|\leq\delta n^{1/m}}e^{-Cu^m}\,du\leq\int_{|u|>M}e^{-Cu^m}\,du
\end{equation*}
and because $e^{-Cu^m}\in L^1(\mathbb{R})$, there exists $M>0$ for which $\mathcal{I}_2<\epsilon/3$. 
Analogously and in view of the fact that $\Re(\beta)>0$, there is $M>0$ for which $\mathcal{I}_3<\epsilon/3$. Selecting $M$ for which these estimates hold and restricting our attention to sufficiently large $n$ for which $M<\delta n^{1/m}$, we move on to estimate $\mathcal{I}_1$.  

Let's first observe that, for all $u$ such that $|u|\leq M<\delta n^{1/m}$,
\begin{equation*}
|g(u)^n-e^{-\beta u^m}|\leq |g(u)|^n+|e^{-\beta u^m}|\leq 2
\end{equation*}
in view of \eqref{easylemmaeq1}. Also, by an appeal to \eqref{type1approx} of Lemma \ref{typeslemma}, for any $u\in[-M,M]$,
\begin{eqnarray*}
\lefteqn{|n((\Gamma (un^{-1/m})-i\alpha un^{-1/m})+\beta u^m|}\\
&=& n|\Gamma(un^{-1/m})-i\alpha un^{-1/m}+\beta (un^{-1/m})^m|\leq nB|un^{-1/m}|^{m+1}\\
&&\hspace{8cm}=B|u|^mn^{-1/m}
\end{eqnarray*}
and so 
\begin{equation*}
\lim_{n\rightarrow\infty}n(\Gamma (un^{-1/m})-i\alpha un^{-1/m})=-\beta u^m. 
\end{equation*}
 Therefore, for each such $u$,
\begin{equation*}
\lim_{n\rightarrow\infty}\left|g(u)^n-e^{-\beta u^m}\right|\\
=\lim_{n\rightarrow\infty}\left|e^{n(\Gamma(un^{-1/m})-i\alpha un^{-1/m})}-e^{-\beta u^m}\right|=0.
\end{equation*}
Because $[-M,M]$ is a set of finite measure, an appeal to the bounded convergence theorem gives $N>(M/\delta)^m$ for which $\mathcal{I}_1\leq \epsilon/3$ for all $n>N$. Combining
the estimates for $\mathcal{I}_1$, $\mathcal{I}_2$ and $\mathcal{I}_3$ gives the desired result. 
\end{proof}

\begin{lemma}\label{hardlemma22}
Let $\nu:\mathbb{R}\rightarrow \mathbb{C}$ be analytic on a neighborhood of a point $\xi_0$ such that $|\nu(\xi_0)|=1$. Assume that $\xi_0\sim(2;2)$ with associated constants $\alpha$ and $\beta$. Let $K\subseteq \mathbb{R}$ be a compact set. Then for all $\epsilon>0$ there is a $\delta>0$ and a natural number $N$ such that 
\begin{equation}
\left|\frac{n^{1/2}}{2\pi}\int_{|\xi-\xi_0|\leq \delta}\nu(\xi)^ne^{-i(xn^{1/2}+\alpha n)\xi}d\xi-e^{-i(xn^{1/2}+\alpha n)\xi_0}\nu(\xi_0)^n H_2^{\beta}(x)\right|<\epsilon
\end{equation}
for all $n>N$ and for all $x\in K$.
\end{lemma}
\begin{proof}
Let $\epsilon>0$, let $K\subseteq\mathbb{R}$ be a fixed compact set and choose $\delta>0$ so that the estimates \eqref{type2approx}, \eqref{type2bound} and \eqref{type2derbound} of Lemma \ref{typeslemma} are valid. Changing variables of integration we write

\begin{eqnarray*}
 \lefteqn{\frac{n^{1/2}}{2\pi}\int_{|\xi-\xi_0|\leq \delta}\nu(\xi)^ne^{-i(xn^{1/2}+\alpha n)\xi}d\xi=}\\ 
&&e^{-i(xn^{1/2}+\alpha n)\xi_0}\nu(\xi_0)^n \frac{n^{1/2}}{2\pi}\int_{|\xi|\leq\delta}[\nu(\xi_0)^{-1}\nu(\xi+\xi_0)]^ne^{-i(xn^{1/2}+\alpha n)\xi}d\xi.
\end{eqnarray*}
Upon setting 
\begin{equation*}
  \mathcal{D}=\left|\frac{n^{1/2}}{2\pi}\int_{|\xi-\xi_0|\leq \delta}\nu(\xi)^ne^{-i(xn^{1/2}+\alpha n)\xi}d\xi-e^{-i(xn^{1/2}+\alpha n)\xi_0}\nu(\xi_0)^n H_2^{\beta}(x)\right|,
\end{equation*}
we have
\begin{eqnarray*}
\mathcal{D}&\leq&\left|\frac{n^{1/2}}{2\pi}\int_{|\xi|\leq Mn^{-1/2}}[\nu(\xi_0)^{-1}\nu(\xi+\xi_0)]^ne^{-i(xn^{1/2}+\alpha n)\xi}d\xi-H_2^{\beta}(x)\right|\\
&&+n^{1/2}\left|\int_{Mn^{-1/2}<|\xi|\leq \delta}[\nu(\xi_0)^{-1}\nu(\xi+\xi_0)]^ne^{-i(xn^{1/2}+\alpha n)\xi}d\xi\right|,
\end{eqnarray*}
where for now $0<M<\delta n^{1/2}$ and we have used the fact that $|\nu(\xi_0)|=1$. Continuing in this manner,
\begin{eqnarray*}
 \mathcal{D}&\leq& \left|\frac{1}{2\pi}\int_{|u|\leq M}[\nu(\xi_0)^{-1}\nu(un^{-1/2}+\xi_0)]^ne^{-i(x+\alpha n^{1/2})u}du-H_2^{\beta}(x)\right|\\
&&+n^{1/2}\left|\int_{Mn^{-1/2}<|\xi|\leq \delta}[\nu(\xi_0)^{-1}\nu(\xi+\xi_0)]^ne^{-i(xn^{1/2}+\alpha n)\xi}d\xi\right|\\
&\leq& \left|\int_{|u|\leq M}\left([\nu(\xi_0)^{-1}\nu(un^{-1/2}+\xi_0)e^{-i\alpha un^{-1/2}}]^n-e^{-\beta u^2}\right)e^{-ixu}du\right|\\
&&+\left|\int_{|u|>M}e^{-ixu-\beta u^2}du\right|\\
&&+n^{1/2}\left|\int_{Mn^{-1/2}}^{\delta}[\nu(\xi_0)^{-1}\nu(\xi+\xi_0)]^ne^{-i(xn^{1/2}+\alpha n)\xi}d\xi\right|\\
&&+n^{1/2}\left|\int_{-\delta}^{-Mn^{-1/2}}[\nu(\xi_0)^{-1}\nu(\xi+\xi_0)]^ne^{-i(xn^{1/2}+\alpha n)\xi}d\xi\right|\\
&&=:\mathcal{I}_1+\mathcal{I}_2+\mathcal{I}_3+\mathcal{I}_4,
\end{eqnarray*}
where we have made a change of variables and used the definition of $H_2^{\beta}$. We now estimate the terms $\mathcal{I}_i$ for $i=1,2,3,4$. First, using Lemma \ref{Hconverge} we choose $M>0$ so that $\mathcal{I}_2\leq \epsilon/4$ for all $x\in K$. Let us now focus on $\mathcal{I}_3$. We write
\begin{eqnarray*}
\mathcal{I}_3&=& n^{1/2}\left|\int_{Mn^{-1/2}}^{\delta}[\nu(\xi_0)^{-1}\nu(\xi+\xi_0)]^ne^{-i(xn^{1/2}+\alpha n)\xi}d\xi\right|\\
&=&n^{1/2}\left|\int_{Mn^{-1/2}}^{\delta}[\nu(\xi_0)^{-1}\nu(\xi+\xi_0)\exp(-i\alpha+i\xi^2p(\xi))]^ne^{i(n\xi^2p(\xi)-\xi xn^{1/2})}d\xi\right|\\
&=&n^{1/2}\left|\int_{Mn^{-1/2}}^{\delta}g(\xi)^n e^{if_n(\xi)}d\xi\right|,
\end{eqnarray*}
where we have put
\begin{equation*}
g(\xi)=\nu(\xi_0)^{-1}\nu(\xi+\xi_0)\exp(-i\alpha+i\xi^2p(\xi))
\end{equation*}
and
\begin{equation*}
f_n(\xi)=(n\xi^2p(\xi)-\xi xn^{1/2}).
\end{equation*}
We wish to apply Lemma \ref{minbound} to the above integral. Set
\begin{equation*}
B=4\left(1+\int_{\mathbb{R}}D|u|^{k-1}e^{-C|u|^k}du\right),
\end{equation*}
where $C,D\geq0$ are the constants appearing in \eqref{type2bound} and \eqref{type2derbound} of Lemma \ref{typeslemma}.
Since $\xi^2p(\xi)$ is a polynomial with $2$ being the smallest power of its terms, we can further restrict $\delta>0$ so that $f_n''(\xi)\neq 0$ and
\begin{equation*}
c_1 \xi\leq \frac{d}{d\xi}(\xi^2p(\xi))\leq c_2\xi
\end{equation*}
for all $\xi\in[Mn^{-1/2},\delta]$, where $c_1$ and $c_2$ are non-zero real numbers of the same sign. Consequently, we can select $M>0$ and a natural number $N$ so that 
\begin{equation*}
 \inf_{\xi\in[Mn^{1/2},\delta]}|f_n'(\xi)|>\frac{4Bn^{1/2}}{\epsilon}
\end{equation*}
for all $x$ in the compact set $K$ and for all $n>N$. Finally, an application of Lemma \ref{minbound} with the above estimate and a calculation similar to that done in the proof of Lemma \ref{bigolem} shows
\begin{equation*}
\mathcal{I}_3\leq \frac{B}{\inf_{\xi\in[Mn^{1/2},\delta]}|f_n'(\xi)|}< \frac{\epsilon}{4}
\end{equation*}
for all $n>N$ and for all $x\in K$. An analogous argument gives the same estimate for $\mathcal{I}_4$.

Before treating $\mathcal{I}_1$, we fix $M$ as the maximal $M$ for which the above estimates hold simultaneously. In view of \eqref{type2bound} of Lemma \ref{typeslemma}, an analogous argument to that given in the proof of Lemma \ref{easylemma} shows that the absolute value of integrand in $\mathcal{I}_1$ is bounded above by $2$ for all $n$. Furthermore, for any $u\in [-M,M]$,
\begin{eqnarray*}
\lefteqn{|n(\Gamma(un^{-1/2})-i\alpha un^{-1/2})+\beta u^2|}\\
&\leq& n|\Gamma(un^{-1/2})-i\alpha un^{-1/2}+ip(u n^{-1/2}) (un^{-1/2})^2|+|\beta u^2-ip(u n^{-1/2})u^2|\\
&\leq& Bn(u n^{-1/2})^k+u^2|\beta-ip(un^{-1/2})|\\
&\leq& Bu^k n^{1-k/2}+u^2|\beta-ip(un^{-1/2})|,
\end{eqnarray*}
where we have used \eqref{type2approx}. Because $p$ is continuous, $ip(0)=\beta$ and $k>2$, it follows that for all $u\in [-M,M]$,
\begin{equation*}
\lim_{n\rightarrow\infty}|n(\Gamma(un^{-1/2})-i\alpha un^{-1/2})+\beta u^2|=0
\end{equation*}
and hence
\begin{equation*}
 \lim_{n\rightarrow \infty} \left|\left(\nu(\xi_0)^{-1}\nu(un^{-1/2}+\xi_0)e^{-i\alpha un^{-1/2}}\right)^n-e^{-\beta u^2}\right|=0.
\end{equation*}
An appeal to the bounded convergence theorem guarantees that for sufficiently large $n$, 
\begin{eqnarray*}
 \mathcal{I}_1&=&\left|\int_{|u|\leq M}\left([\nu(\xi_0)^{-1}\nu(un^{-1/2}+\xi_0)e^{-i\alpha un^{-1/2}}]^n-e^{-\beta u^2}\right)e^{-ixu}du\right|\\
&\leq& \int_{-M}^{M}\left|\left(\nu(\xi_0)^{-1}\nu(un^{-1/2}+\xi_0)e^{-i\alpha un^{-1/2}}\right)^n-e^{-\beta u^2}\right|du<\epsilon/4
\end{eqnarray*}
for all $x\in \mathbb{R}$ and in particular for all $x\in K$. Finally, from the above arguments we choose $\delta>0$ and a natural number $N$ so that for each $j=1,2,3,4$, $\mathcal{I}_j<\epsilon/4$ for all $n>N$ and for all $x\in K$. Putting these estimates together shows that $\mathcal{D}<\epsilon$ as desired. 
\end{proof}

\begin{lemma}\label{hardlemma2m}
Let $\nu:\mathbb{R}\rightarrow \mathbb{C}$ be analytic on a neighborhood of a point $\xi_0$ such that $|\nu(\xi_0)|=1$. Let $m>2$ and assume that $\xi_0\sim(2;m)$ with associated constants $\alpha$ and $\beta$. Then for all $\epsilon>0$ there is a $\delta>0$ and a natural number $N$ such that 
 \begin{equation}
\left|\frac{n^{1/m}}{2\pi}\int_{|\xi-\xi_0|\leq\delta}\nu(\xi)^n
e^{-ix\xi}d\xi-e^{-ix\xi_0}\nu(\xi_0)^n H_{m}^{\beta}\left(\frac{x-\alpha
n}{n^{1/m}}\right)\right|<\epsilon
\end{equation}
for all $n>N$ and for all real numbers $x$.
\end{lemma}

\noindent The present lemma's proof is analogous to the proof of the previous lemmas in many ways. We will consequently spend less time explaining the order in which we choose our constants.
\begin{proof}
Let $\epsilon>0$ and set
\begin{equation}
 y_n=\frac{x-\alpha n}{n^{1/m}}
\end{equation}
and 
\begin{equation*}
 \mathcal{D}=\left|\frac{n^{1/m}}{2\pi}\int_{|\xi-\xi_0|\leq\delta}\nu(\xi)^n
e^{-ix\xi}d\xi-e^{-ix\xi_0}\nu(\xi_0)^n H_{m}^{\beta}\left(\frac{x-\alpha
n}{n^{1/m}}\right)\right|.
\end{equation*}

Since $\xi_0\sim(2;m)$, we choose $\delta>0$ so that the estimates \eqref{type2approx}, \eqref{type2bound} and \eqref{type2derbound} of Lemma \ref{typeslemma} are valid, and the inequality
\begin{equation}\label{inf2m}
 \left|\frac{d^2}{d\xi^2}\xi^m p(\xi)\right|\geq B^2|\xi|^{m-2}
\end{equation}
holds for all $-\delta\leq\xi\leq\delta$, where $B>0$. Using Proposition \ref{Hconverge}, we now choose $M>0$ such that
\begin{equation}\label{tail2m}
\left|\int_{|u|>M}e^{-iyu-\beta u^m}du\right|<\epsilon/4
\end{equation}
for all $y\in\mathbb{R}$ and
\begin{equation}
 \frac{8}{BM^{m/2-1}}\left(1+\int_{\mathbb{R}}D|u|^{k-1}e^{-Cu^k}du\right)<\epsilon/4,
\end{equation}
where $B$ was defined above and $C$ and $D$ are the constants appearing in \eqref{type2derbound}. As in the last lemma, we write
\begin{eqnarray*}
\mathcal{D}&=&\left|\frac{n^{1/m}}{2\pi}\int_{-\delta}^{\delta}\nu(\xi+\xi_0)^n
e^{-ix(\xi+\xi_0)}d\xi-e^{-ix\xi_0}\nu(\xi_0)^n H_{m}^{\beta}(y)\right|\\
&\leq& \left|\int_{|u|\leq M}\left([\nu(\xi_0)^{-1}\nu(un^{-1/m}+\xi_0)e^{-\alpha un^{-1/m}}]^n-e^{-\beta u^m}\right)e^{-iyu}du\right|\\
&&+\left|\int_{|u|>M}e^{-iyu-\beta u^m}du\right|+n^{1/m}\left|\int_{Mn^{-1/m}}^{\delta}[\nu(\xi_0)^{-1}\nu(\xi+\xi_0)]^n
e^{-ix\xi}d\xi\right|\\
&&+n^{1/m}\left|\int_{-\delta}^{-Mn^{-1/m}}[\nu(\xi_0)^{-1}\nu(\xi+\xi_0)]^n
e^{-ix\xi}d\xi\right|\\
&&=:\mathcal{I}_1+\mathcal{I}_2+\mathcal{I}_3+\mathcal{I}_4.
\end{eqnarray*}
Now we estimate the terms $\mathcal{I}_j$ for $j=1,2,3,4.$ Already from \eqref{tail2m}, we know that $\mathcal{I}_2<\epsilon/4$ for all $x\in\mathbb{R}$. We have
\begin{eqnarray*}
\mathcal{I}_3&=&n^{1/m}\left|\int_{Mn^{-1/m}}^{\delta}[\nu(\xi_0)^{-1}\nu(\xi+\xi_0)e^{-i\alpha\xi+i\xi^m p(\xi)}]^n
e^{-i(x+\alpha n)\xi-in\xi^m p(\xi)}d\xi\right|\\
&=&n^{1/m}\left|\int_{Mn^{-1/m}}^{\delta}g(\xi)^ne^{if_n(\xi)}d\xi\right|,
\end{eqnarray*}
where
\begin{equation*}
 g(\xi)=[\nu(\xi_0)^{-1}\nu(\xi+\xi_0)e^{-i\alpha\xi+i\xi^m p(\xi)}]
\end{equation*}
and 
\begin{equation*}
f_n(\xi)=-[(x+\alpha n)\xi-in\xi^m p(\xi)].
\end{equation*}
With the aim of applying Lemma \ref{minbound}, we use \eqref{inf2m} and observe that on the interval $[Mn^{-1/m},\delta]$
\begin{equation*}
\inf|f_n''(\xi)|\geq\inf nB^2|\xi|^{m-2}\geq nB^2 |Mn^{-1/m}|^{m-2}=(n^{1/m}B M^{m/2-1})^2>0.
\end{equation*}
The application of the lemma is therefore justified and we can use \eqref{type2bound} and \eqref{type2derbound} to see that
\begin{eqnarray*}
\mathcal{I}_3&\leq& \frac{8n^{1/m}}{\sqrt{(n^{1/m}B M^{m/2-1})^2}}(\|g^n\|_{\infty}+\|ng'g^{n-1}\|_1)\\
&\leq& \frac{8}{CM^{m/2-1}}\left(1+\int_{\mathbb{R}}D|u|^{k-1}e^{-Cu^k}du\right)<\epsilon/4
\end{eqnarray*}
for all $x\in \mathbb{R}$. A similar calculation gives the same estimate for $\mathcal{I}_4$. 

To estimate $\mathcal{I}_1$, we essentially repeat the argument given in the proof of the previous lemma. Again, the integrand is bounded in absolute value by $2$ for all $n$. Using \eqref{type2approx}, we observe that for any $u\in[-M,M]$
\begin{equation*}
 \lim_{n\rightarrow \infty} \left|\left(\nu(\xi_0)^{-1}\nu(un^{-1/m}+\xi_0)e^{-\alpha un^{-1/m}}\right)^n-e^{-\beta u^m}\right|=0.
\end{equation*}
Therefore, the bounded convergence theorem gives a natural number $N$ for which $\mathcal{I}_1\leq \epsilon/4$ for all $n>N$ and for all $x\in \mathbb{R}$. Combining our estimates finishes the proof.
\end{proof}

\noindent For the remainder of this section, we focus on local limit theorems. The first theorem, Theorem \ref{weakresult22}, focuses on the case in which $\phi^{(n)}$ is approximated locally on its packets by linear combinations of the attractors $H_2^{\beta}$. The second theorem, Theorem \ref{stongestresult}, isolates the second conclusion of Theorem \ref{mainstrong}. The results of both theorems are then used to prove Theorem \ref{mainweak}. 

\begin{theorem}\label{weakresult22}
Let $\phi:\mathbb{Z}\rightarrow\mathbb{C}$ have admissible support and suppose that $\sup_{\xi}|\hat{\phi}(\xi)|=1$. Under Convention \ref{constantsconvention}, suppose that $m=2$ and for some $q=1,2,\dots,R$, $\beta_q$ is purely imaginary. Then, to each $\alpha_q$, there exists subcollections $\xi_{j_1},\xi_{j_2},\dots,\xi_{j_{r(q)}}$ and $\beta_{j_1},\beta_{j_2},\dots,\beta_{j_{r(q)}}$, such that 
\begin{equation}\label{convolpower22}
\phi^{(n)}(\lfloor xn^{1/2}+\alpha_q n\rfloor)=\sum_{l=1}^{r(q)}n^{-1/2}e^{-i(\lfloor xn^{1/2}+\alpha_q
n\rfloor) \xi_{j_l}}\hat{\phi}(\xi_{j_l})^n H_2^{\beta_{j_l}}(x)+o(n^{-1/2})
\end{equation}  
uniformly on any compact set. 
\end{theorem}
\begin{proof}

Let $\epsilon>0$ and $K\subseteq\mathbb{R}$ be a compact set. In view of Convention \ref{constantsconvention}, it follows from our hypotheses that $Q=R$ and therefore $\Omega=\{\xi_1,\xi_2,\dots,\xi_R\}$.  We note that the corresponding drift constants $\alpha_1,\alpha_2,\dots,\alpha_R$ need not be distinct.

Let $\alpha_q$ be a member of the above collection and let $\{j_1,j_2,\dots,j_{r(q)}\}$ be the increasing subcollection  of $\{1,2,\dots, R\}$ for which $\alpha_{j_l}=\alpha_q$ for $l=1,2,\dots, r(q)$. Also, set $\Upsilon_q=\{1,2,\dots,R\}\setminus\{j_1,j_2,\dots,j_{r(q)}\}$. It is of course possible that $\Upsilon_q$ is empty. For example, it might be the case that $1=r(q)=R$ and, in this case, \eqref{convolpower22} consists only of the single attractor $H_2^{\beta_1}$. In fact, this is precisely the situation exemplified in the introduction in which $\phi$ was defined by \eqref{ex1def} (see also Subsection \ref{ex1}).

We divide $T$ into subintervals: For
$l=1,2,\dots,R$, define $I_l=[\xi_l-\delta_l,\xi_l+\delta_l]\subseteq T$ where $\delta_l>0$ are to be defined shortly; for now, let's require them to be sufficiently small to ensure that the intervals $I_l$, for $l=1,2,\dots,R$, are disjoint. In view of \eqref{phiextdef}, put $J=T\setminus\cup I_l$ and write
\begin{eqnarray}\label{split22}
\lefteqn{n^{1/2}\phi_{\mbox{\scriptsize{e}}}(n,xn^{1/2}+\alpha_q n)}\\
&=&\frac{n^{1/2}}{2\pi}\int_T
\hat{\phi}(\xi)^ne^{-i(xn^{1/2}+\alpha_q n)\xi}d\xi \nonumber\\
&=&\sum_{l=1}^{R}\frac{n^{1/2}}{2\pi}\int_{I_l}
\hat{\phi}(\xi)^ne^{-i(xn^{1/2}+\alpha_q n)\xi}d\xi \nonumber+\frac{n^{1/2}}{2\pi}\int_J \hat{\phi}(\xi)^ne^{-i(xn^{1/2}+\alpha_q
n)\xi}d\xi\nonumber\\
&=&\sum_{l=1}^{R}\mathcal{I}_l+\mathcal{E}.
\end{eqnarray}

We treat the integrals $\mathcal{I}_l$ in the two cases separately. First, we consider
$\mathcal{I}_{l}$ for $l\in\Upsilon_q$. Here we show that $\mathcal{I}_l$ can be made arbitrarily small (depending on $x$ and $n$) because $\alpha_q\neq \alpha_l$. If $\xi_l\sim(2;2)$, let $\gamma_l,k_l$ and $p_l(\xi)$ be associated as per Definition \ref{types}. We have 
\begin{eqnarray*}
|\mathcal{I}_l|&=&\left|\frac{n^{1/2}}{2\pi}\int_{I_l}
\hat{\phi}(\xi)^ne^{-i(xn^{1/2}+\alpha_q n)\xi}d\xi\right|\\
&=&\left|\frac{n^{1/2}e^{-i(xn^{1/2}-\alpha_q n)\xi_l}\hat{\phi}(\xi_l)^n}{2\pi
}\int_{|\xi|\leq\delta}\left[\hat{\phi}(\xi_l)^{-1}\hat{\phi}(\xi_l+\xi)\right]^n
e^{-i(xn^{1/2}+\alpha_q n)\xi}d\xi\right|\\
&\leq& n^{1/2}\left|\int_{|\xi|\leq\delta}g_l(\xi)^n e^{if_{n,l}(\xi)}d\xi\right|,\\
\end{eqnarray*}
where
\begin{equation*}
 g_l(\xi)=[\hat{\phi}^{-1}(\xi_l)\hat{\phi}(\xi_l+\xi)e^{-i\alpha_l\xi+i\xi^2p_l(\xi)}]
\end{equation*}
and
\begin{equation*}
 f_{n,l}(\xi)=-n[(xn^{-1/2}+\alpha_q-\alpha_l)\xi+\xi^2p_l(\xi)].
\end{equation*}
Now choose $\delta_l>0$ so that, on the interval $[-\delta_l,\delta_l]$,  $g_l(\xi)$ satisfies \eqref{type2bound} and \eqref{type2derbound} for some $C_l,D_l>0$, 
\begin{equation*}
 f_{n,l}''(\xi)=-n\frac{d^2}{d\xi^2}\xi^2p_l(\xi)\neq 0
\end{equation*}
and
\begin{equation}\label{estbq}
 B_l\leq \left|\alpha_l-\alpha_q-\frac{d}{d\xi}\xi^2p_l(\xi)\right|
\end{equation}
for some $B_l>0$. For the first property our choice of $\delta_l$ was made using Lemma \ref{typeslemma} and the assumption that $\xi_l\sim(2;2)$. For the second two properties we used that fact that $\xi^2p_l(\xi)$ is a polynomial with $2$ being the smallest power of its terms and $\alpha_l\neq\alpha_q$. We can therefore apply Lemma \ref{minbound}. This gives
\begin{eqnarray*}
|\mathcal{I}_l|&\leq& \frac{8n^{1/2}}{\inf_{\xi}|f_{n,l}'(\xi)|}(\|g_l\|_\infty+\|ng_l'g_l^{n-1}\|)\\
&\leq& \frac{8}{\inf_{\xi}|(x-n^{1/2}(\alpha_l-\alpha_q-\frac{d}{d\xi}\xi^2p_l(\xi))|}\left(1+\int_{\mathbb{R}}D_l|\xi|^{k_l}e^{-C_l\xi^{k_l}}d\xi\right)\\
&\leq& \frac{M_l}{\inf_{\xi}|(x-n^{1/2}(\alpha_l-\alpha_q-\frac{d}{d\xi}\xi^2p_l(\xi))|}\\
\end{eqnarray*}
for some $M_l>0$ and where the above infima are taken over the interval $[-\delta_l,\delta_l]$. Using the estimate \eqref{estbq} and recalling that $x$ lives inside the compact set $K$, we can choose a natural number $N_l$ so that
\begin{equation*}
 \inf_{\xi}|(x-n^{1/2}(\alpha_l-\alpha_q-\frac{d}{d\xi}\xi^2p_l(\xi))|>\frac{M_l (R+1)}{\epsilon}
\end{equation*}
for all $n>N_l$ and for all $x\in K$. Consequently, 
\begin{equation}\label{QR122}
 |\mathcal{I}_l|\leq \frac{M_l}{M_l (R+1)/\epsilon}=\epsilon/(R+1)
\end{equation}
for all $n>N_l$ and for all $x\in K$.

If instead $\xi_l\sim(1;2)$, by an appeal to Lemma \ref{easylemma}, we choose $\delta_l>0$ and a natural number $N_l$ so that
\begin{eqnarray*}
|\mathcal{I}_l|&\leq &\epsilon/{2(R+1)}+|e^{-i(xn^{1/2}+\alpha_q
n)\xi_l}\hat\phi(\xi_l)^n H_2^{\beta_l}(x+(\alpha_q-\alpha_l)n^{1/2})|\\
&\leq&\epsilon/{2(R+1)}+|H_2^{\beta_l}(x+(\alpha_q-\alpha_l)n^{1/2})|\\
\end{eqnarray*}
for all $n>N_l$ and for all $x\in \mathbb{R}$. However, as we remarked earlier $H_2^{\beta_l}$ is the heat kernel evaluated at
complex time $\beta_l$. Since $\Re(\beta_l)>0$ in this case and
$\alpha_q\neq\alpha_l$ we may increase our natural number $N_l$ to ensure that 
\begin{equation*}
 |H_2^{\beta_l}(x+(\alpha_q-\alpha_l)n^{1/2})|\leq \epsilon/ {2(R+1)}
\end{equation*}
for any $n>N_l$ and for all $x$ in the compact set $K$. These estimates together give 
\begin{equation}\label{R1R22}
 |\mathcal{I}_l|\leq \epsilon/ {2(R+1)}+\epsilon/ {2(R+1)} =\epsilon/(R+1)
\end{equation}
for all $n>N_l$ and for all $x\in K$.

In the remaining estimates of $\mathcal{I}_l$ for $l=j_1,j_2,\dots,j_{r(q)}$, we recall that $\alpha_q=\alpha_l$. If $\xi_l\sim(2;2)$, we appeal to Lemma \ref{hardlemma22}. From this we choose $\delta_l>0$ and a natural number $N_l$ such that
\begin{equation}\label{1Q122}
|\mathcal{I}_l-e^{-i(xn^{1/2}+\alpha_q
n)\xi_l}\hat{\phi}(\xi_l)^n H_2^{\beta_l}(x)|\leq \epsilon/(R+1)
\end{equation}
for all $n>N_l$ and for all $x\in K$. If instead $\xi_l\sim(1;2)$, we appeal to Lemma \ref{easylemma}
and chose $\delta_l>0$ and $N_l$, a natural number, such that
\begin{equation}\label{QQ122}
|\mathcal{I}_l-e^{-i(xn^{1/2}+\alpha_q
n)\xi_l}\hat{\phi}(\xi_l)^n H_2^{\beta_l}(x)|\leq \epsilon/(R+1)
\end{equation}
for all $n>N_l$ and for all $x\in \mathbb{R}$. In particular we have this estimate uniform for all $x\in K$.

After fixing our collection of $\delta_l$'s in the above arguments, the set $J$ becomes fixed. We therefore set $s=\sup_{\xi\in J}|\hat{\phi}(\xi)|<1$ and note that $|\mathcal{E}|\leq
n^{1/2}s^n$. Thus we may choose a natural number $N_0$ such that $|\mathcal{E}|<\epsilon/(R+1)$ for all $n>N_0$ and for all $x\in K$.

At last, we choose $N$ to be the maximum of $N_l$ for $l=0,1,\dots, R$. Combining the estimates \eqref{split22}, \eqref{QR122}, \eqref{R1R22}, \eqref{1Q122} and \eqref{QQ122} yields
\begin{eqnarray*}
 \lefteqn{\left|n^{1/2}\phi_{\mbox{\scriptsize{e}}}(n,xn^{1/2}+\alpha_q n)-\sum_{l\in\{j_1,j_2\dots,j_{r(q)}\}}e^{-i(xn^{1/2}+\alpha_q
n)\xi_l}\hat{\phi}(\xi_l)^n H_2^{\beta_l}(x)\right|}\\
&\leq&\sum_{l\in\{j_1,j_2,\dots,j_{r(q)}\}}\left|\mathcal{I}_l-e^{-i(xn^{1/2}+\alpha_q
n)\xi_l}\hat{\phi}(\xi_l)^n H_2^{\beta_l}(x)\right|+\sum_{l\in\Upsilon_q}|\mathcal{I}_l|+\mathcal{E}\\
&<& \frac{(R+1)\epsilon}{R+1}=\epsilon
\end{eqnarray*}
for any $n>N$ and for all $x\in K$. We have shown that
\begin{equation}\label{tildephiweak}
\phi_{\mbox{\scriptsize{e}}}(n, xn^{1/2}+\alpha_q n)=\sum_{l=1}^{r(q)}n^{-1/2}e^{-i( xn^{1/2}+\alpha_q
n) \xi_{j_l}}\hat{\phi}(\xi_{j_l})^n H_2^{\beta_{j_l}}(x)+o(n^{-1/2})
\end{equation}
uniformly for $x$ in any compact set $K$. 

To complete the proof of the theorem we need to replace the argument $xn^{1/2}+\alpha_q
n$ by an integer in \eqref{tildephiweak}; this is precisely where the floor function comes in. Let $K\subseteq \mathbb{R}$ be compact, set 
\begin{equation*}
y(x,n)=\frac{\lfloor \alpha_q n+xn^{1/2}\rfloor-\alpha_q n}{n^{1/2}},
\end{equation*}
and observe that $|x-y(n,x)|\leq n^{-1/2}$. Let $F\supseteq K$ be any compact set for which $y(x,n)\in F$ for all $x\in K$ and all natural numbers $n$. By Proposition \ref{Hanalytic}, each function $H_2^{\beta_{j_l}}$  is uniformly continuous on $F$ and therefore, for any $x\in K$, we have
\begin{eqnarray}\label{Hcontweak}\nonumber
\lefteqn{\sum_{l=1}^{r(q)}n^{-1/2}e^{-i( \lfloor xn^{1/2}+\alpha_q
n\rfloor) \xi_{j_l}}\hat{\phi}(\xi_{j_l})^n H_2^{\beta_{j_l}}(y(x,n))}\\
&=&\sum_{l=1}^{r(q)}n^{-1/2}e^{-i( \lfloor xn^{1/2}+\alpha_q
n\rfloor) \xi_{j_l}}\hat{\phi}(\xi_{j_l})^n H_2^{\beta_{j_l}}(x)+o(n^{-1/2}).
\end{eqnarray}
The result now follows from \eqref{tildephiweak}, \eqref{Hcontweak}  and the observation that 
\begin{equation*}
 \phi^{(n)}(\lfloor xn^{1/2}+\alpha_q
n\rfloor)=\phi_{\mbox{\scriptsize{e}}}(n,\lfloor xn^{1/2}+\alpha_q n\rfloor).
\end{equation*}
\end{proof}

\begin{theorem}\label{stongestresult}
Let $\phi:\mathbb{Z}\rightarrow\mathbb{C}$ have admissible support and suppose that $\sup|\hat{\phi}(\xi)|=1$. Under Convention \ref{constantsconvention}, additionally assume that $m>2$ or $\Re(\beta_q)>0$ for all $q=1,2,\dots, R$ (this precisely the hypothesis \eqref{mainstronghypoth} of Theorem \ref{mainstrong}). Then
\begin{equation}
\phi^{(n)}(x)=\sum_{q=1}^{R}n^{-1/m}e^{-ix\xi_q}\hat{\phi}(\xi_q)^n H_m^{\beta_q}\left(\frac{x-\alpha_qn}{n^{1/m}}\right)+o(n^{-1/m})
\end{equation}  
uniformly in $\mathbb{Z}$.
\end{theorem}

\begin{proof}
In view of Proposition \ref{typesprop} and under Convention \ref{constantsconvention}, our hypotheses guarantee that either $m>2$ or, in the case that $m=2$, $\xi_q\sim(1;2)$ for each $\xi_q\in\Omega$. Consequently to each point $\xi_q\in\Omega$ of order $m$ we may apply either Lemma \ref{easylemma} or Lemma \ref{hardlemma2m}.

As in the proof of the previous theorem we divide $T$ into subintervals. For
$q=1,2,\dots,Q$, let $I_q=[\xi_q-\delta_q,\xi_q+\delta_q]$ for values of $\delta_q>0$ to be chosen later (but small enough to ensure that the $I_q$'s are disjoint) and set $J=T\setminus\cup I_q$. We again define $\phi_{\mbox{\scriptsize{e}}}$ by \eqref{phiextdef} and write
\begin{eqnarray*}
n^{1/m}\phi_{\mbox{\scriptsize{e}}}(n,x)&=&\frac{n^{1/m}}{2\pi}\int_T
\hat{\phi}(\xi)^ne^{-ix\xi}d\xi \\
&=&\sum_{q=1}^{Q}\frac{n^{1/m}}{2\pi}\int_{I_q}
\hat{\phi}(\xi)^ne^{-ix\xi}d\xi+\frac{n^{1/m}}{2\pi}\int_J \hat{\phi}(\xi)^ne^{-ix\xi}d\xi.\\
\end{eqnarray*}
Therefore,
\begin{eqnarray}\label{splitm}\nonumber
 \lefteqn{\left|n^{1/m}\phi_{\mbox{\scriptsize{e}}}(n,x)-\sum_{q=1}^{R}e^{-ix\xi_q}\hat{\phi}(\xi_q)^n H_m^{\beta_q}\left(\frac{x-\alpha_qn}{n^{1/m}}\right)\right|}\\\nonumber
&\leq& \sum_{q=1}^{R}\left|\frac{n^{1/m}}{2\pi}\int_{I_q}
\hat{\phi}(\xi)^ne^{-ix\xi}d\xi-e^{-ix\xi_q}\hat{\phi}(\xi_q)^n H_m^{\beta_q}\left(\frac{x-\alpha_qn}{n^{1/m}}\right)\right| \\
&&+\sum_{q=R+1}^{Q} n^{1/m}\left|\frac{1}{2\pi}\int_{I_q}
\hat{\phi}(\xi)^ne^{-ix\xi}d\xi\right|+\left|\frac{n^{1/m}}{2\pi}\int_J \hat{\phi}(\xi)^ne^{-ix\xi}d\xi\right|.
\end{eqnarray}

As we previously noted, for $q=1,2,\dots,R$, we apply either Lemma \ref{easylemma} or Lemma \ref{hardlemma2m}. We can therefore choose a natural number $N_q$ and fix $\delta_q>0$ so that 
\begin{equation}\label{estdiffm}
 \left|\frac{n^{1/m}}{2\pi}\int_{I_q}
\hat{\phi}(\xi)^ne^{-ix\xi}d\xi-e^{-ix\xi_q}\hat{\phi}(\xi_q)^n H_m^{\beta_q}\left(\frac{x-\alpha_qn}{n^{1/m}}\right)\right|<\frac{\epsilon}{(Q+1)}
\end{equation} for all $n>N_q$ and for all $x\in\mathbb{R}$.

In the case that $q=R+1,R+2,\dots, Q$, we appeal to Lemma \ref{bigolem} and choose $\delta_q>0$ and a natural number $N_q$ such that
\begin{equation*}
 \left|\frac{1}{2\pi}\int_{I_q}
\hat{\phi}(\xi)^ne^{-ix\xi}d\xi\right|\leq \frac{C_q}{n^{1/m_q}}
\end{equation*}
for some $C_q>0$ and for all $n>N_q$ and $x\in \mathbb{R}$. Using the fact that $m>m_q$ we can adjust the value of $N_q$ so that
\begin{equation}\label{estmqm}
n^{1/m}\left|\frac{1}{2\pi}\int_{I_q}
\hat{\phi}(\xi)^ne^{-ix\xi}d\xi\right|\leq \frac{C_q}{n^{1/m_q-1/m}}<\frac{\epsilon}{(Q+1)}
\end{equation}
for all $n>N_q$ and for all $x\in\mathbb{R}$.

Finally, as in the proof of the last theorem, we set $s=\inf_J |\hat{\phi}|<1$ and observe that the last term in \eqref{splitm} is bounded by $n^{1/m}s^n$. We therefore select a natural number $N_0$ such that
\begin{equation}\label{estJ2m}
 \left|\frac{n^{1/2}}{2\pi}\int_J \hat{\phi}(\xi)^ne^{-ix\xi}d\xi\right|\leq n^{1/m}s^n<\frac{\epsilon}{(Q+1)}
\end{equation}
for all $n>N_0$ and for all $x\in\mathbb{R}$.

Let us choose $N$ to be the maximum $N_q$ for $q=0,1,\dots,Q$. Upon combining the estimates \eqref{estdiffm}, \eqref{estmqm}, \eqref{estJ2m} and \eqref{splitm} we have
\begin{equation}\label{strongresultextension}
  \left|n^{1/m}\phi_{\mbox{\scriptsize{e}}}(n,x)-\sum_{q=1}^{R}e^{-ix\xi_q}\hat{\phi}(\xi_q)^n H_m^{\beta_q}\left(\frac{x-\alpha_qn}{n^{1/m}}\right)\right|<\epsilon
\end{equation}
for all $n>N$ and for all $x\in \mathbb{R}$. In particular, \eqref{strongresultextension} holds for all $x\in\mathbb{Z}$ and for such $x$, $\phi_{\mbox{\scriptsize{e}}}(n,x)=\phi^{(n)}(x)$. This is our desired result.
\end{proof}

\begin{proof}[Proof of Theorem \ref{mainweak}]
Let $K$ be a compact set. Assuming that $\phi$ satisfies the hypotheses of the theorem, we adopt Convention \ref{constantsconvention} by virtue of Proposition \ref{typesprop}. There are two distinct possibilities pertaining to the constants $m$ and $\beta_1,\beta_2,\dots,\beta_R$: they satisfy the hypotheses of Theorem \ref{weakresult22} or they satisfy the hypotheses of Theorem \ref{stongestresult}. A moment's thought shows that the hypotheses of Theorem \ref{weakresult22} and the hypotheses of Theorem \ref{stongestresult} are indeed mutually exclusive and collectively exhaustive. If the case at hand is the former there is nothing to prove for $m=2$ and the desired result is precisely the conclusion of Theorem \ref{weakresult22}. We therefore address the latter case.  

Let $\alpha_q\in\{\alpha_1,\alpha_2,\dots,\alpha_R\}$ and, exactly as was done in the proof of Theorem \ref{weakresult22}, define $\{j_1,j_2,\dots,j_r(q)\}\subseteq\{1,2,\dots, R\}$ and $\Upsilon_q$.   Observe that \eqref{strongresultextension} is uniform in $\mathbb{R}$ and we can therefore write
\begin{eqnarray*}
\lefteqn{\phi_{\mbox{\scriptsize{e}}}(n,\alpha_qn+xn^{1/m})}\\
&=&\sum_{l=1}^{R}n^{-1/m}e^{-i(\alpha_qn+xn^{-1/m})\xi_l}\hat{\phi}(\xi_l)^n H_m^{\beta_l}\left(\frac{(\alpha_q-\alpha_l)n+xn^{1/m}}{n^{1/m}}\right)+o(n^{-1/m})\\
&=& \sum_{l\in\{j_1,j_2,\dots,j_{r(q)}\}}n^{-1/m}e^{-i(\alpha_qn+xn^{-1/m})\xi_l}\hat{\phi}(\xi_l)^n H_m^{\beta_l}(x)\\
&&+ \sum_{l\in\Upsilon_q}n^{-1/m}e^{-i(\alpha_qn+xn^{-1/m})\xi_l}\hat{\phi}(\xi_l)^n H_m^{\beta_l}((\alpha_q-\alpha_l)n^{1-1/m}+x)+o(n^{-1/m}).\\
&=&\sum_{l=1}^{r(q)}n^{-1/m}e^{-i(\alpha_qn+xn^{-1/m})\xi_{j_l}}\hat{\phi}(\xi_{j_l})^n H_m^{\beta_{j_l}}(x)+\sum_{l\in\Upsilon_q}S_l(n,x)+o(n^{-1/m}).
\end{eqnarray*}
Upon requiring $x\in K$, we consider the summands $S_l(n,x)$ for $l\in\Upsilon_q$. In the case that $\Re(\beta_l)>0$, we have
\begin{eqnarray*}
|S_l(n,x)|&=&|n^{-1/m}e^{-i(\alpha_qn+xn^{-1/m})\xi_l}\hat{\phi}(\xi_l)^n H_m^{\beta_l}((\alpha_q-\alpha_l)n^{1-1/m}+x)|\\
&=&n^{-1/m}|H_m^{\beta_l}((\alpha_q-\alpha_l)n^{1-1/m}+x))|\\
&\leq& n^{-1/m} C_l\exp(-B_l((\alpha_q-\alpha_l)n^{1-1/m}+x)^{m/(m-1)})=o(n^{-1/m}).
\end{eqnarray*}
If it is the case that $\Re(\beta_l)=0$, we must have $m>2$. Appealing to Proposition \ref{Hassymptoticprop}, we conclude that
\begin{eqnarray*}
\lefteqn{|S_l(n,x)|}\\
&\leq& n^{-1/m}\left(\frac{A}{|(\alpha_q-\alpha_l)n^{1-1/m}+x)|^{\frac{m-2}{2(m-1)}}}+\frac{B}{|(\alpha_q-\alpha_l)n^{1-1/m}+x)|}\right)\\
&=&o(n^{-1/m}).
\end{eqnarray*}
Combining the above estimates shows that, for all $x\in K$,
\begin{equation*}
\phi_{\mbox{\scriptsize{e}}}(n,\alpha_qn+xn^{1/m})=\sum_{l=1}^{r(q)}n^{-1/m}e^{-i(\alpha_qn+xn^{-1/m})\xi_{j_l}}\hat{\phi}(\xi_{j_l})^n H_m^{\beta_{j_l}}(x)
+o(n^{-1/m}).
\end{equation*}
To complete the proof, it remains to replace the argument $\alpha_qn+xn^{1/m}$ by the integer $\lfloor\alpha_qn+xn^{1/m}\rfloor$ in the equation above. This can be done easily by making an argument analogous to that given in the last paragraph of the proof to Theorem \ref{weakresult22}. From this, the desired result follows without trouble.
\end{proof}

\section{The lower bound of $\|\phi^{(n)}\|_{\infty}$}\label{lowerboundsec}

In this section we complete the proof of Theorem \ref{mainbound}. 
\begin{lemma}\label{vandermondelemma}
Let $\zeta_1,\zeta_2,\cdots,\zeta_r\in (-\pi,\pi]$ be distinct, let $B>0$ and define
\begin{equation}\label{vandermondematrix}
V=
\begin{pmatrix}
1 &  1  & \ldots & 1\\
e^{-i\zeta_1}  &  e^{-i\zeta_2} & \ldots & e^{-i\zeta_r}\\
e^{-i2\zeta_1}  &  e^{-i2\zeta_2} & \ldots & e^{-i2\zeta_r}\\
\vdots & \vdots & \ddots & \vdots\\
e^{-i(r-1)\zeta_1}  &  e^{-i(r-1)\zeta_2} & \ldots & e^{-i(r-1)\zeta_r}\\
\end{pmatrix}.
\end{equation}
Then there is a number $C>0$ such that for any $\rho,\sigma\in\mathbb{C}^r$ with $\|\rho\|>B$ and $\sigma=V\rho$, we have $|\sigma_j|>3C$ for some $j=1,2,\dots,r$. Here $\|\cdot\|$ denotes the usual norm on $\mathbb{C}^r$.
\end{lemma}
\begin{proof}
The matrix V in \eqref{vandermondematrix} is known as Vandermonde's matrix. It is a routine exercise in linear algebra to show that
\begin{equation*}
\det(V)=\prod_{1\leq l<k\leq r} (e^{-i\zeta_k}-e^{-i\zeta_l}).
\end{equation*}
Noting that $e^{-i\zeta_1}, e^{-i\zeta_1}, \dots, e^{-i\zeta_r}$ are all distinct we conclude that $V$ is invertible. The proof now follows immediately from the estimate
\begin{equation*}
\|\rho\|\leq\|V^{-1}\|\|\sigma\|.
\end{equation*}
\end{proof}

\begin{proof}[Proof of Theorem \ref{mainbound}]
Let $\phi:\mathbb{Z}\rightarrow\mathbb{C}$ have admissible support. As Theorem
\ref{firsthalfmainbound} gave the upper bound
\begin{equation*}
A^{-n}\|\phi^{(n)}\|_{\infty}\leq C'n^{1/m}
\end{equation*}
for some $C'>0$, our job is establish the lower bound
\begin{equation*}
Cn^{1/m}\leq A^{-n}\|\phi^{(n)}\|_{\infty} 
\end{equation*}
for some $C>0$. This is done with the help of our local limit theorems.

As we noted in the proof of Theorem \ref{firsthalfmainbound}, it suffices to
assume that $A=\sup_{\xi}|\hat{\phi}(\xi)|=1$. We adopt Convention \ref{constantsconvention} by virtue of Proposition \ref{typesprop} and note that $m\geq2$, defined by \eqref{mdef}, is that which appears in both Theorem \ref{mainweak} and Theorem \ref{firsthalfmainbound}. In view of Theorem \ref{mainweak}, set $\alpha=\alpha_1$, $r=r(1)$ and correspondingly take $\xi_{j_1},\xi_{j_2},\dots,\xi_{j_{r}}\in(-\pi,\pi]$ and $\beta_{j_1},\beta_{j_2},\dots,\beta_{j_{r}}$ for which \eqref{mainweakconvergenceconclusion} holds. For notational convenience, set $b_l=\beta_{j_l}$ and $\zeta_l=\xi_{j_l}$ for $l=1,2,\dots, r$ and note that the points $\zeta_1,\zeta_2,\dots,\zeta_r\in(-\pi,\pi]$ are distinct. In this notation, \eqref{mainweakconvergenceconclusion} is the assertion that
\begin{equation}\label{lowerboundlimiteq}
 \phi^{(n)}(\lfloor\alpha n+xn^{1/m}\rfloor)=\sum_{l=1}^{r}n^{-1/m}e^{-i\lfloor\alpha n+xn^{1/m}\rfloor\zeta_l}\hat{\phi}(\zeta_l)^m H_m^{b_l}(x)+o(n^{-1/m})
\end{equation}
uniformly for $x$ in a compact set.

Appealing to Proposition \ref{Hanalytic}, we know that each function
$H_m^{b_1}$ is non-zero and continuous for $l=1,2,\dots, r$. In particular, there exists $B>0$ and an interval $I=[a,b]$ such that $|H_m^{b_1}(x)|\geq
B$ for all $x\in I$. Define $V$ by \eqref{vandermondematrix} and let $C>0$ as guaranteed by Lemma \ref{vandermondelemma}. Set
\begin{equation}
f(n,x)=\sum_{l=1}^{r}e^{-i(\alpha n+xn^{1/m})\zeta_l}\hat{\phi}(\zeta_l)^n
H_m^{b_l}(x)
\end{equation}
and
\begin{equation}
\sigma_k(n,x)=\sum_{l=1}^{r}e^{-ik\zeta_l}e^{-i(\alpha n+xn^{1/m})\zeta_l}\hat{\phi}(\zeta_l)^n
H_m^{b_l}(x)
\end{equation}
for $k=0,1,\dots,r-1$. Since each function $H_m^{b_l}$ is continuous on $\mathbb{R}$ it is uniformly continuous on $[a-r,b+r]\supseteq I$. Consequently, we may choose a natural number $N$ for which
\begin{equation}\label{fsigmarelation}
|f(n,x+kn^{-1/m})-\sigma_k(n,x)|<C
\end{equation}
for all $n\geq N$, $k=0,1,\dots,r-1$ and $x\in I$.  By possibly further increasing $N$ we can also guarantee that for any $n\geq N$ there is $x_0\in I$ such that
$\alpha n+x_0 n^{1/m} $ is an integer and for which $x_0+kn^{-1/m}\in I $ for
all $k=0,1,\dots, r-1$. We observe that for any such $k$,  $(\alpha n+(x_0+kn^{-1/m}) n^{1/m})$ is also an integer.

Now for any $n\geq N$, let $x_0\in I$ be as guaranteed in the previous paragraph. Observe that
\begin{eqnarray*}
\begin{pmatrix}
\sigma_0(n,x_0)\\
\sigma_1(n,x_0)\\
\vdots \\
\sigma_{(R-1)}(n,x_0)\\
\end{pmatrix}
&=&
\begin{pmatrix}
1 &  1  & \ldots & 1\\
e^{-i\zeta_1}  &  e^{-i\zeta_2} & \ldots & e^{-i\zeta_r}\\
e^{-i2\zeta_1}  &  e^{-i2\zeta_2} & \ldots & e^{-i2\zeta_r}\\
\vdots & \vdots & \ddots & \vdots\\
e^{-i(r-1)\zeta_1}  &  e^{-i(r-1)\zeta_2} & \ldots & e^{-i(r-1)\zeta_r}\\
\end{pmatrix}
\begin{pmatrix}
\rho_1(n,x_0) \\
\rho_2(n,x_0)\\
\vdots\\
\rho_r(n,x_0)\\
\end{pmatrix},\\
\end{eqnarray*}
where
\begin{equation*} 
\rho_l(x_0,n)=e^{-i(\alpha n+x_0n^{1/m})\zeta_l}\hat{\phi}(\zeta_l)^n
H_m^{b_l}(x_0)
\end{equation*}
for $l=1,2,\dots,r$. Because $x_0\in I$, $|\rho_1(x_0,n)|=|H_m^{b_1}(x_0)|>B$ and therefore
\begin{equation*}
\|(\rho_1(n,x_0),\rho_2(n,x_0),\dots,\rho_r(n,x_0))^{\top}\|>B.
\end{equation*}
Appealing to Lemma \ref{vandermondelemma}, there is some $k\in\{0,1,2,\dots,r-1\}$ such that \break$|\sigma_k(n,x_0)|>3C$ and so by \eqref{fsigmarelation}, $|f(n,x_0+kn^{-1/m})|>2C$.

We have shown that there is a natural number $N$, a closed interval $I$ and a constant $C>0$ such that for any $n\geq N$
\begin{equation}\label{supsumbound}
\sup\left|\sum_{l=1}^{r}e^{-i(\alpha n+xn^{1/m})\zeta_l}\hat{\phi}(\zeta_l)^n
H_m^{b_l}(x)\right|>2C,
\end{equation}
where the above supremum is taken over the set
\begin{equation*}
 \{x:x\in I\mbox{ and }(\alpha n+xn^{-1/m})\in \mathbb{Z}\}.
\end{equation*}
Combining \eqref{lowerboundlimiteq} and \eqref{supsumbound} we
conclude that
\begin{equation}\label{supphin}
\sup_{x\in\mathbb{Z}}|\phi^{(n)}(x)|\geq Cn^{-1/m}
\end{equation}
for all $n>N$. The result now follows from the observation that $\phi^{(n)}\neq
0$ for all $n\leq N$ and so, by possibly adjusting $C$, \eqref{supphin}
must hold for all $n$. 
\end{proof}

\section{Concentration of mass}\label{masssec}

\noindent In this section we complete the proof of Theorem \ref{mainstrong}. Recall that the theorem has two conclusions, the second of which is the subject of Theorem \ref{stongestresult} and was already shown in the previous section. The first conclusion, \eqref{mainstrongsetconclusion}, remains to be shown.

\begin{proof}[Proof of Theorem \ref{mainstrong}]
We assume that $\phi$ satisfies the hypotheses of the theorem. By Theorem \ref{stongestresult},
\begin{equation}\label{limittheoremconcentrationofmass}
\phi^{(n)}(x)=\sum_{q=1}^{R}n^{-1/m}e^{-ix\xi_q}\hat{\phi}(\xi_q)^n H_m^{\beta_q}\left(\frac{x-\alpha_qn}{n^{1/m}}\right)+o(n^{-1/m}),
\end{equation}
where the limit is uniform for $x\in \mathbb{Z}$ and the collections $\xi_1,\xi_2,\dots,\xi_R\in(-\pi,\pi]$, $\alpha_1,\alpha_2,\dots,\alpha_R$ and $\beta_1,\beta_2,\dots,\beta_R$ are those set by Convention \ref{constantsconvention}.

Using Theorem \ref{mainbound} we choose $C>0$ for which the estimate \eqref{mainboundsupeq} holds. Considering all possibilities of $\beta_q$ and $m$ above, we can choose $M>0$ such that
\begin{equation*}
|H_m^{\beta_q}(y)|< C/(2R)
\end{equation*}
for all $|y|> M$ and for all $q=1,2,\dots, R$. This can be done by using \eqref{Hassymptoticnicecase} or the conclusion of Proposition \ref{Hassymptoticprop}. Now let $K=[-M,M]$ and observe that, for any $q=1,2,\dots,R$, 
\begin{equation}\label{massconcentrationbound1}
 \left|n^{-1/m}e^{-ix\xi_q}\hat{\phi}(\xi_q)^n H_m^{\beta_q}\left(\frac{x-\alpha_qn}{n^{1/m}}\right)\right|<\frac{Cn^{-1/m}}{2R}
\end{equation}
whenever $(x-\alpha_q n)/n^{1/m}>M$ or equivalently $x\notin \alpha_q n+ Kn^{1/m}$.  Further, by combining \eqref{limittheoremconcentrationofmass} and \eqref{massconcentrationbound1} there is some natural number $N$ such that
\begin{equation*}
|\phi^{(n)}(x)|< C  n^{-1/m}
\end{equation*}
for all $x\notin \cup_q(\alpha_q n+ Kn^{1/m})$ and $n>N$. Thus by Theorem \ref{mainbound}, the supremum $\|\phi^{(n)}\|_{\infty}$ must be attained on the set $(\cup_q(\alpha_q n+ Kn^{1/m}))\cap\mathbb{Z}$ for all $n>N$. Lastly, observe that by enlarging the compact set $K$, the above dependence on $N$ can be removed. This completes the proof.

\end{proof}

\section{Examples}\label{exsec}
In this final section, we consider three examples to illustrate our results. We begin by considering a complex valued function on $\mathbb{Z}$ whose convolution powers consist of two waves drifting apart. This example cannot be treated by the results of Schoenberg, Greville or Thom\'{e}e. 

\subsection{Two Airy functions with drift}\label{airyex}

Consider the function $\phi:\mathbb{Z}\rightarrow\mathbb{C}$ defined by
\begin{equation*}
\phi(0)=\frac{3}{8}\hspace{.5cm}\phi(\pm 2)=-\frac{1}{4}\hspace{.5cm}\phi(\pm 3)=\frac{i}{3}\hspace{.5cm}\phi(\pm 4)=\frac{1}{16}
\end{equation*}
and $\phi(x)=0$ otherwise. The convolution powers, $\phi^{(n)}$, exhibit two distinct packets drifting apart, each with a rate of $2n$ from $x=0$. Figure \ref{fig:ex3_waves} illustrates this behavior. 

\begin{figure}[h!]
\centering\includegraphics[width=4.7in]{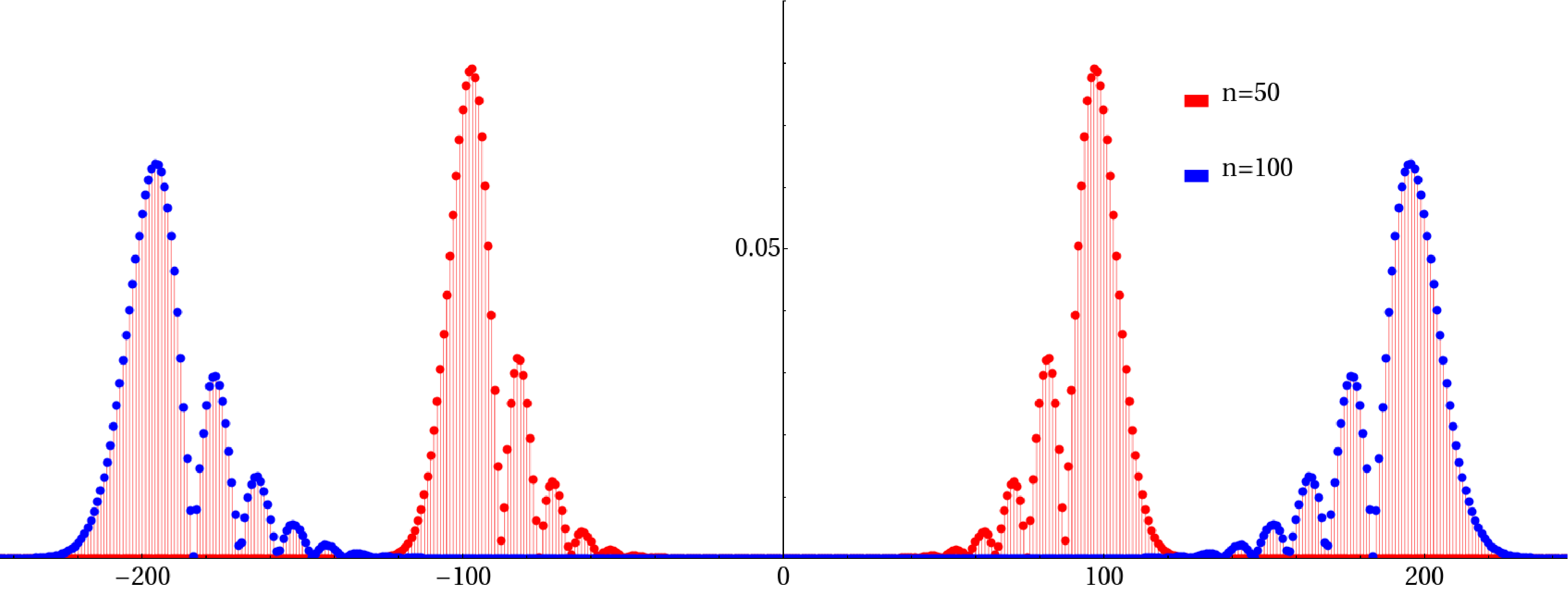}
\caption{$|\phi^{(n)}|$ for $n=50,100$}
\label{fig:ex3_waves}
\end{figure}

\noindent The Fourier transform of $\phi$ is given by
\begin{equation*}
\hat{\phi}(\xi)=\frac{3}{8}-\frac{1}{2}\cos(2\xi)+\frac{2i}{3}\cos(3\xi)+\frac{1}{8}\cos(4\xi).
\end{equation*}
Here, $\sup|\hat{\phi}|=1$ and is attained only at $\xi_1=\pi/2$ and $\xi_2=-\pi/2$ in $(-\pi,\pi]$. It follows that
\begin{equation*}
\log\left(\frac{\hat{\phi}(\xi\pm\pi/2)}{\hat{\phi}(\pm\pi/2)}\right)=\pm2i\xi\mp\frac{5i}{3}\xi^3-\frac{7}{3}\xi^4(1+o(1)) \mbox{ as }\xi\rightarrow 0
\end{equation*}
and so $\alpha_1=2$, $\alpha_2=-2$, $\beta_1=5i/3$, $\beta_2=-5i/3$ and $m=m_1=m_2=3$. In view of Theorem \ref{mainstrong} (or Theorem \ref{stongestresult}),
\begin{eqnarray}\label{doubleairyapprox1}\nonumber
\phi^{(n)}(x)&=& n^{-1/3}e^{-ix\pi/2}H^{\frac{5i}{3}}_3\left(\frac{x-2n}{n^{1/3}}\right)+n^{-1/3}e^{ix\pi/2}H^{\frac{-5i}{3}}_3\left(\frac{x+2n}{n^{1/3}}\right)+o(n^{-1/3})\\\nonumber
&=&(5n)^{-1/3}(i)^x\left[(-1)^xH_3^{\frac{i}{3}}\left(\frac{x-2n}{(5n)^{1/3}}\right)+H_3^{\frac{i}{3}}\left(-\frac{x+2n}{(5n)^{1/3}}\right)\right]+o(n^{-1/3})\\\nonumber
&=&(5n)^{-1/3}(i)^x\left[(-1)^x\mbox{Ai}\left(\frac{x-2n}{(5n)^{1/3}}\right)+\mbox{Ai}\left(-\frac{x+2n}{(5n)^{1/3}}\right)\right]+o(n^{-1/3})\\
&=&f(n,x)+o(n^{-1/3})
\end{eqnarray}
uniformly for $x\in\mathbb{Z}$, where $\mbox{Ai}$ denotes the standard Airy function. To appreciate Theorems \ref{mainstrong} and \ref{mainweak}, we consider $\phi^{(n)}(x)$ for $n=10000$ near the right packet ($19700\leq x\leq 20150$) corresponding to drift constant $\alpha_1=\pi/2$. Figure \ref{fig:ex3simulation} shows the graph of $\Re(\phi^{(n}(x))$ and Figure \ref{fig:ex3twoairy} shows the approximation, $f(n,x)$ defined by \eqref{doubleairyapprox1}.

\begin{figure}[h!]
\centering\includegraphics[width=5in]{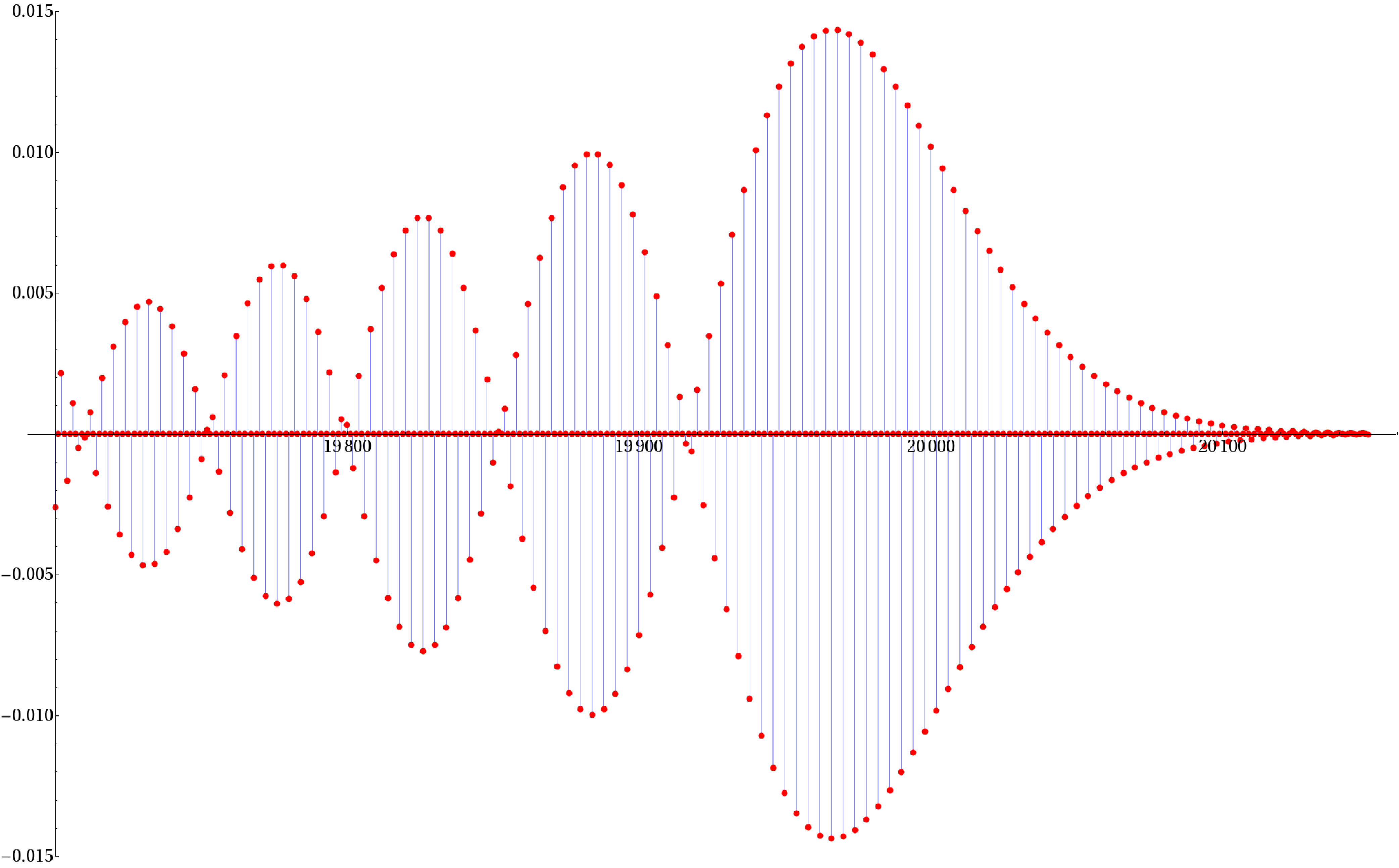}
\caption{ $\Re(\phi^{(n)})$ for $n=10000$}
\label{fig:ex3simulation}
\end{figure}
\begin{figure}[h!]
\centering\includegraphics[width=5in]{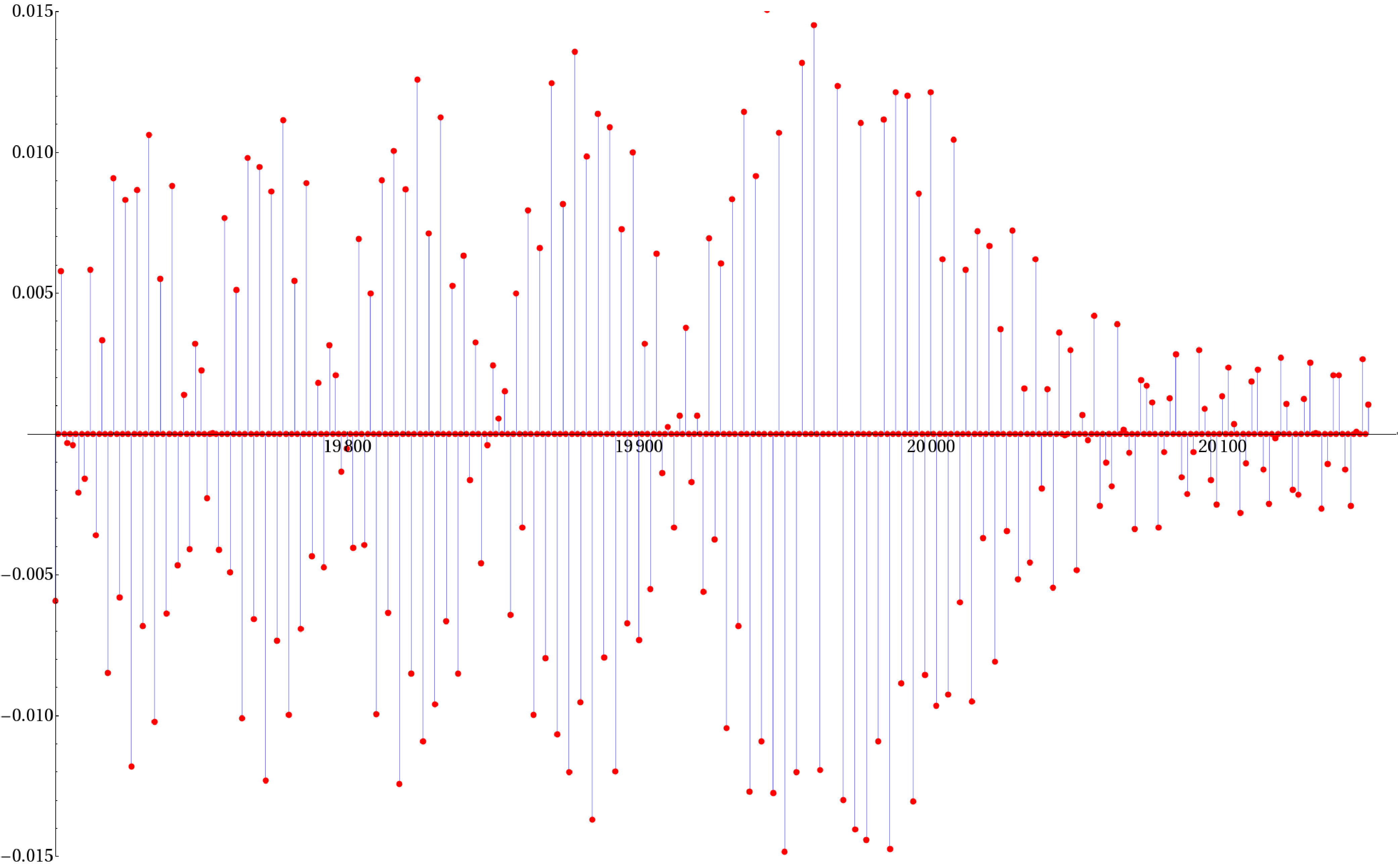}
\caption{ $\Re(f(n,x))$ for $n=10000$}
\label{fig:ex3twoairy}
\end{figure}
\begin{figure}[h!]
\centering\includegraphics[width=5in]{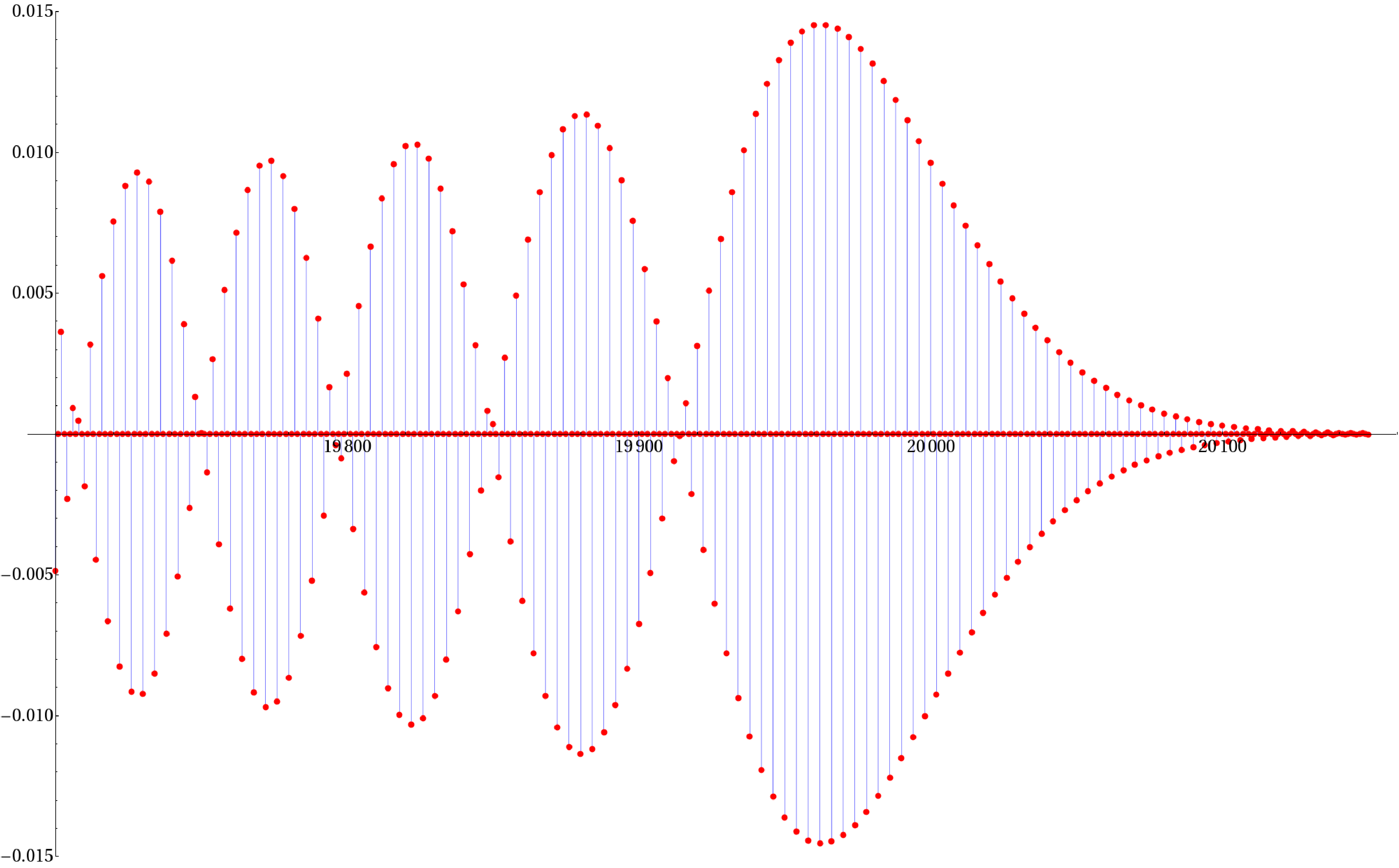}
\caption{$\Re(g(n,x))$ for $n=10000$}
\label{fig:ex3oneairy}
\end{figure}

\break

\noindent What appears to be noise in Figure \ref{fig:ex3twoairy} is the oscillatory tail of the term \break $(5n)^{-1/3}(i)^x\mbox{Ai}\left(-\frac{x+2n}{(5n)^{1/3}}\right)$ in \eqref{doubleairyapprox1}. Removing this term, we consider
\begin{eqnarray*}
 g(n,x)&=&f(n,x)-(5n)^{-1/3}(i)^x\mbox{Ai}\left(-\frac{x+2n}{(5n)^{1/3}}\right)\\
&=&(5n)^{-1/3}(-i)^x\mbox{Ai}\left(\frac{x-2n}{(5n)^{1/3}}\right).
\end{eqnarray*}
Upon choosing $\alpha_1=\pi/2$, an appeal to Theorem \ref{mainweak} gives the approximation
\begin{eqnarray*}
\phi^{(n)}(\lfloor 2n+zn^{1/3}\rfloor)&=&n^{-1/3}e^{-(i\lfloor 2n+zn^{1/3}\rfloor \pi/2)}H_3^{\frac{5i}{3}}(z)+o(n^{-1/3})\\
&=&(5n)^{-1/3}(-i)^{\lfloor2n+zn^{1/3}\rfloor}\mbox{Ai}\left(\frac{z}{5^{1/3}}\right)+o(n^{-1/3})
\end{eqnarray*}
uniformly for $z$ in any compact set; here, $\xi_{j_1}=\xi_{1}=\pi/2$ and $\beta_{j_1}=\beta_{1}=5/3$. For such $z$, it follows that
\begin{equation*}
\phi^{(n)}(\lfloor 2n+zn^{1/3}\rfloor)=g(n,\lfloor 2n+zn^{1/3}\rfloor)+o(n^{-1/3})
\end{equation*}
from which we see that $g$ is essentially the approximation yielded by Theorem \ref{mainweak}. As Figure \ref{fig:ex3oneairy} shows, $g(n,x)$ is a much better approximation to $\phi^{(n)}(x)$ at $n=10000$ for $19700\leq x\leq20150$. 

\subsection{Heat kernel at purely imaginary time}\label{ex1}

We return to the example given in the introduction and justify the claims made
therein. Let $\phi$ be given by \eqref{ex1def}. A quick computation shows that 
\begin{equation*}
\hat{\phi}(\xi)=1-\frac{i}{2}\sin^2(\xi/2)-\sin^{4}(\xi/2),
\end{equation*}
where the supremum of $|\hat{\phi}|$ on the interval $(-\pi,\pi]$ is only
attained at $\xi_1=0$. In the notation of Proposition \ref{typesprop}, we write
\begin{equation*} 
\Gamma(\xi)=\log\left(\frac{\hat{\phi}(\xi)}{\hat{\phi}(0)}\right)=-i\xi^2(\frac{1
}{8}-\frac{1}{96}\xi^2)-\frac{7}{128}\xi^4+\sum_{l=5}^\infty a_l\xi^l
\end{equation*}
on a neighborhood of $0$ and so $m=m_1=2$, $\alpha_1=0$ and
$\beta_1=i/8$ in view of Convention \ref{constantsconvention}. By Theorem \ref{mainbound}, there are constants $C,C'>0$ such
that
\begin{equation}
Cn^{1/2}\leq \|\phi^{(n)}\|_{\infty}\leq C'n^{1/2}.
\end{equation}
By Theorem \ref{mainweak} and using \eqref{heatker} we may also conclude that
\begin{eqnarray*}
\phi^{(n)}(\lfloor xn^{1/2}\rfloor)&=&n^{-1/2}H_2^{i/8}(x)+o(n^{-1/2})\\
&=&\frac{n^{-1/2}}{\sqrt{4\pi i/8}}e^{-8|x|^2/4i}+o(n^{-1/2}),
\end{eqnarray*}
where the limit is uniform for $x$ in any compact set. 

\subsection{A real-valued function supported on three points}\label{ex3}

\noindent In the article \cite{DSC1}, Example 2.4 and Proposition 2.5 therein described the asymptotic behavior of the convolution powers of an arbitrary real valued function $\phi$ supported on three (consecutive) points.  In the notation of the proposition we define $\phi$ by 
\begin{equation}
 \phi(0)=a_0,\hspace{.1cm}\phi(\pm 1)=a_{\pm}\mbox{ and }\phi=0\mbox{ otherwise},
\end{equation}
where $a_0,a_{+},a_{-}\in\mathbb{R}$. As in \cite{DSC1}, we also assume that $a_0>0$ and that $a_+\neq 0$ or $a_{-}\neq 0$; this assumption guarantees that $\phi$ has admissible support. Proposition 2.5 of \cite{DSC1} describes the asymptotic behavior of $\phi^{(n)}$ for all values of $a_0, a_{\pm}$ except the special case in which $a_{+}a_{-}<0$ and $4|a_{+}a_{-}|=a_0|a_{+}+a_{-}|.$ Theorem \ref{mainstrong} of the present article allows us to treat this final case with ease. 
\begin{proposition}
Let $\phi$ be as above and assume additionally that $a_{+}a_{-}<0$ and $4|a_{+}a_{-}|=a_0|a_{+}+a_{-}|.$ If $a_{+}+a_{-}>0$ then
\begin{equation}\label{threepointconclusion1}
\phi^{(n)}(x)=n^{-1/3}A^n H_3^{\beta}\left(\frac{x-\alpha n}{n^{1/3}}\right)+o(A^n n^{-1/3})
\end{equation}
uniformly for $x\in\mathbb{Z}$, where $A=a_0+a_{+} +a_{-}$, $\alpha=(a_{+}-a_{-})/A$ and $\beta=i(\alpha-\alpha^3)/6$.
 
If  $a_{+}+a_{-}<0$ then
\begin{equation}
\phi^{(n)}(x)=n^{-1/3}A^n e^{-ix\pi} H_3^{\beta}\left(\frac{x-\alpha n}{n^{1/3}}\right)+o(A^n n^{-1/3})
\end{equation}
uniformly for $x\in\mathbb{Z}$, where $A=a_0-a_{+} -a_{-}$, $\alpha=(a_{-}-a_{+})/A$ and $\beta=i(\alpha-\alpha^3)/6$.

In either case, there is a compact set $K$ for which the $\|\phi^{(n)}\|_\infty$ is attained on the set $(\alpha n+Kn^{1/3})$.
\end{proposition}
\begin{proof}
We may write
\begin{equation*}
\hat{\phi}(\xi)=a_0+a_{+}e^{i\xi}+a_{-}e^{-i\xi}=a_0+ (a_{+}+a_{-})\cos(\xi)+i(a_{+}-a_{-})\sin(\xi).
\end{equation*}
Under the assumption that $4|a_{+}a_{-}|=a_0|a_{+}+a_{-}|$ and $a_{+}+a_{-}>0$, it was shown in \cite{DSC1} that $|\hat{\phi}|$ is maximized only at $0=\xi_1\in(-\pi,\pi]$ and in which case this maximum takes the value $A=a_0+a_{+} +a_{-}$. 

Set $\psi(x)=\phi(x)/A$. It follows immediately that $A^n\psi^{(n)}(x)=\phi^{(n)}(x)$ and $\sup|\hat{\psi}|=1$ which is taken only at $\xi_1=0$. In the notation of Proposition \ref{typesprop} we have
\begin{eqnarray*}
\Gamma(\xi)&=&\log\left(\frac{\hat{\psi}(\xi)}{\hat{\psi}(0)}\right)=i\left(\frac{(a_{+}-a_{-})}{a_0+a_{+}+a_{-}}\right)\xi\\
&&-\frac{i}{6}\left(\frac{(a_{+}-a_{-})(a_{0}^2-a_0 a_{+}-a_0 a_{-}-8a_{+}a_{-})}{(a_0+a_{+}+a_{-})^3}\right)\xi^3
-\gamma\xi^4+\sum_{l=5}^{\infty}a_l\xi^l
\end{eqnarray*}
on a neighborhood of $0$, where $\gamma>0$. Setting $\alpha=(a_{+}-a_{-})/A$ and using the fact that $4|a_{+}a_{-}|=a_0|a_{+}+a_{-}|$, we write
\begin{equation}\label{threepointstaylor}
\Gamma(\xi)=i\alpha\xi-\frac{i}{6}(\alpha-\alpha^3)\xi^3-C\xi^4+\sum_{l=5}^{\infty}a_l\xi^l
\end{equation}
on a neighborhood of $0$.  By a quick inspection of \eqref{threepointstaylor} it is clear that $\psi$ meets the hypotheses of Theorem \ref{mainstrong} with $m=m_1=3$, $\alpha=\alpha_1$ and $\beta_1=i(\alpha-\alpha^3)/6$. Therefore
\begin{equation}\label{psithreepoints1}
\psi^{(n)}(x)=n^{-1/3}H_3^{\beta}\left(\frac{x-\alpha n}{n^{1/3}}\right)+o(n^{-1/3})
\end{equation}
uniformly for $x\in \mathbb{Z}$. The limit \eqref{threepointconclusion1} follows immediately by multiplying \eqref{psithreepoints1} by $A^n$. An appeal to \eqref{mainstrongsetconclusion} of Theorem \ref{mainstrong} shows that $\|\psi^{(n)}\|_{\infty}$ and hence $\|\phi^{(n)}\|_{\infty}$ is indeed attained on the set $(\alpha n+Kn^{1/3})$ for some compact set $K$.

In the case that $a_{+}+a_{-}<0$ it was shown in \cite{DSC1} that $|\hat{\phi}|$ attains its only maximum at $\xi_1=\pi\in(-\pi,\pi]$. Upon setting $A=a_0-a_{+} -a_{-}$, $\psi(x)=\phi(x)/A$ and considering the Taylor expansion of $\log(\hat{\psi}(\xi+\xi_1)/\hat{\psi}(\xi_1))$, the result follows by an argument similar to that given for the previous case.
\end{proof}

\noindent\small{\textbf{Acknowledgements}\hspace{.25cm}The authors would like to thank the referee for many useful suggestions and comments. }

\end{document}